\documentclass[11pt]{article}
	
	\usepackage{algorithm}
	\usepackage{algpseudocode}
    \usepackage{url}
    \usepackage{verbatim}
    \usepackage[titletoc]{appendix}
    \usepackage{graphicx}
	\usepackage{amsmath}
	\usepackage{amsthm}
	\usepackage{amsfonts}
	\usepackage{graphicx}
    \usepackage{nicefrac}
    \usepackage{longtable}
    \usepackage{color}
    \usepackage{graphicx, amssymb,graphics}
 	\usepackage{subfigure}
 	\usepackage{epstopdf}

\newtheorem{rem}{Remark}[section]
\newtheorem{thm}{Theorem}[section]

\newtheorem{lem}{Lemma}[section]
\newtheorem{defn}{Definition}[section]
\newtheorem{prop}{Proposition}[section]

\newcommand{\bfx}{\boldsymbol{x}}

\newcommand{\hf}{\nicefrac{1}{2}}
\newcommand{\nrm}[1]{\left\| #1 \right\|}
\newcommand{\ciptwo}[2]{\left\langle #1 , #2 \right\rangle}
\newcommand{\cipgen}[3]{\left\langle #1 , #2 \right\rangle_{#3}}

\newcommand{\diff}{\mathrm{d}}

\newcommand{\dx}{\diff\bfx}
\newcommand{\dS}{\diff S}
\newcommand{\eipx}[2]{\left[ #1 , #2 \right]_{\rm x}}
\newcommand{\eipy}[2]{\left[ #1 , #2 \right]_{\rm y}}
\newcommand{\eipvec}[2]{\left[ #1 , #2 \right]}

\newcommand\be {\begin{equation}}
\newcommand\ee {\end{equation}}

	\title{A uniquely solvable and positivity-preserving finite difference scheme for the Flory-Huggins-Cahn-Hilliard equation with dynamical boundary condition}

	\date{\today}

\begin{document}
	
\author{
Yunzhuo Guo \thanks{School of Mathematical Sciences, Beijing Normal University, Beijing 100875, P.R. China (yunzguo@mail.bnu.edu.cn)}
\and
Cheng Wang\thanks{Department of Mathematics, The University of Massachusetts, North Dartmouth, MA  02747, USA (cwang1@umassd.edu)}
	\and
Steven M. Wise\thanks{Department of Mathematics, The University of Tennessee, Knoxville, TN 37996, USA (Corresponding Author: swise1@utk.edu)}
\and	
Zhengru Zhang\thanks{Laboratory of Mathematics and Complex Systems, Beijing Normal University, Beijing 100875, P.R. China (zrzhang@bnu.edu.cn)}
}

 	\maketitle
	\numberwithin{equation}{section}
	
	\begin{abstract}	 
In this paper we propose and analyze a finite difference numerical scheme for the Flory-Huggins-Cahn-Hilliard equation with dynamical boundary condition. The singular logarithmic potential is included in the Flory-Huggins energy expansion. Meanwhile, a dynamical evolution equation for the boundary profile corresponds to a lower-dimensional singular energy potential, coupled with a non-homogeneous boundary condition for the phase variable. In turn, a theoretical analysis for the coupled system becomes very challenging, since it contains nonlinear and singular energy potentials for both the interior region and on the boundary. In the numerical design, a convex splitting approach is applied to the chemical potential associated with the energy both at the interior region and on the boundary: implicit treatments for the singular and logarithmic terms, as well as the surface diffusion terms, combined with an explicit treatment for the concave expansive term. In addition, the discrete boundary condition for the phase variable is coupled with the evolutionary equation of the boundary profile. The resulting numerical system turns out to be highly nonlinear, singular and coupled. A careful finite difference approximation and convexity analysis reveals that such a numerical system could be represented as a minimization of a discrete numerical energy functional, which contains both the interior and boundary integrals. More importantly, all the singular terms correspond to a discrete convex functional. As a result, a unique solvability and positivity-preserving analysis could be theoretically justified, based on the subtle fact that the singular nature of the logarithmic terms around the singular limit values prevent the numerical solutions (at both the interior region and on the boundary section) reaching these values. The total energy stability analysis could be established by a careful estimate over the finite difference  inner product.  Some numerical results are presented in this article, which demonstrate the robustness of the proposed numerical scheme.

	\bigskip

\noindent
{\bf Key words and phrases}: 
Cahn-Hilliard equation, Flory-Huggins energy potential, dynamical boundary condition, positivity-preserving, total energy stability 

\noindent
{\bf AMS subject classification}: \, 35K35, 35K55, 49J40, 65M06, 65M12	
\end{abstract}
	


\section{Introduction} 
The Cahn-Hilliard (CH) equation~\cite{allen79, cahn58} is a fundamental model that describes phase separation processes in binary mixtures. Typically, this equation can be written as follows 
	\begin{align}
\phi_t & = \Delta \mu , 
 	\label{equation-CH-1} 
 	\\
\mu  &= \varepsilon^{-1} F' (\phi) - \varepsilon \Delta \phi, 
	\label{equation-CH-2}
	\end{align}
where $\phi$ is the required phase separation and $\mu$ is the chemical potential. $F$ is a double-well energy and $\varepsilon$ is an interface thickness parameter. Assuming homogeneous Neumann boundary conditions, as is standard, the free energy of the system, namely,
	\begin{equation}
E(\phi) = \int_\Omega \Big( \varepsilon^{-1} F (\phi) + \frac{\varepsilon}{2} | \nabla \phi |^2 \Big) \dx ,
	\label{energy-0} 
	\end{equation}
is dissipated along solution trajectories. Recently, this choice of boundary conditions has come under scrutiny, as it ignores the influence of the solid wall on internal dynamics. More importantly, the homogeneous Neumann boundary condition is known to be invalid in some important  applications, such as the contact line problem. Aiming to account for the possible short-range interactions of the solid wall, physicists have derived several types of dynamic boundary conditions by introducing surface free energy. We analyze a numerical scheme for a dynamic boundary conditions model proposed by Liu and Wu~\cite{LiuC2019} and derived using an energy-variational principle. In more details, the Liu-Wu model reads
	\begin{align} 
& \phi_t = \Delta \mu ,  \quad \mu = \varepsilon^{-1} F' (\phi) - \varepsilon \Delta \phi  ,\quad 
 \mbox{in $\Omega$} ,
	\label{equation-CHDBC-1} 
	\\
& \partial_{n} \mu = 0 ,  \quad \left.\phi\right|_\Gamma = \psi ,  \quad 
 \mbox{on $\partial \Omega = \Gamma$} ,  
	\label{equation-CHDBC-2}   
	\\
& \psi = \Delta_\Gamma \nu ,  \quad \nu = - \kappa \Delta_\Gamma \psi + \varepsilon^{-1} G' (\psi) + \varepsilon \partial_{n} \phi ,  \quad \mbox{on $\partial \Omega = \Gamma$} ,
	\label{equation-CHDBC-3}    
	\end{align}  
where $\Omega$ is a bounded domain with Lipschitz continuous boundary~$\Gamma:=\partial \Omega$; $\nu$ stands for the surface chemical potential; $\Delta_\Gamma$ is the Laplace-Beltrami operator on $\Gamma$; the constant~$\kappa$~is the surface diffuse interface thickness parameter; and $\varepsilon$ is an bulk diffuse interface thickness parameter. Effectively, periodic boundary conditions are employed for the surface diffusion problem. Owing to the no mass flux $\partial_n\mu = 0$ for the bulk problem and the periodic boundary conditions for the surface problem, one will observe that no mass is exchanged between the surface and bulk regions, which leads to respective mass conservation of $\phi$ in $\Omega$ and $\psi$ on $\Gamma$. Effectively, in the Liu-Wu model, the bulk boundary condition is given by a lower dimensional CH type equation that is coupled with the surface normal derivative term~$\partial_{n} \phi$.

The bulk and surface energies are given by
	\begin{equation} 
E_{\rm bulk} (\phi) = \int_\Omega \Big( \frac{\varepsilon}{2} | \nabla \phi |^2 + \varepsilon^{-1} F (\phi) \Big) \dx , \quad E_{\rm surf} (\psi) = \int_\Gamma \,\Big( \frac{\kappa}{2} | \nabla_\Gamma \psi |^2 + \varepsilon^{-1} G (\psi) \Big) \,  \dS . 
	\label{energy-CH-1}
	\end{equation} 
To close the model, suitable energy functions $F$~and~$G$~are needed in the system. We will focus on the Florry-Hugins logarithmic energy, which is also known as the regular solution model. In other words, we set $F(\phi) = R(\phi)$ and $G(\psi) = R(\psi)$, where
	\begin{equation}
R (c) = (1+ c) \ln (1+c) + (1-c) \ln (1-c) - \frac{\theta_0}{2} c^2 .
	\end{equation}
The regular solution model, which in many physical applications is considered more realistic than a polynomial double-well energy density, has the added benefit that it keeps the solutions ``positive." In other words, the solutions are expected to satisfy $-1<\phi,\psi<1$.

Mathematically, the unique existence of both weak and strong solutions has been proved in~\cite{LiuC2019}, and a different way to structure the weak solution was proposed recently \cite{Garcke20}. The dissipation rate can be calculated as follows. First, define
	\[
E_{\rm tot}(\phi) := E_{\rm bulk} (\phi)  + E_{\rm surf} (\psi).
	\]
Then, it follows that
	\begin{align}
\diff_t E_{\rm tot}(\phi) & = \int_\Omega\left\{\varepsilon \nabla\phi \cdot \nabla\partial_t\phi +\varepsilon^{-1}F'(\phi)\partial_t\phi   \right\} \dx \nonumber
	\\
& \quad + \int_\Gamma \left\{\kappa \nabla_\Gamma\psi \cdot \nabla_\Gamma\partial_t\psi +\varepsilon^{-1}G'(\psi)\partial_t\psi   \right\} \dS \nonumber
	\\
& = \int_\Omega\left\{-\varepsilon \Delta \phi  +\varepsilon^{-1}F'(\phi)  \right\} \partial_t\phi  \,\dx + \int_\Gamma \varepsilon\partial_n\phi \partial_t\phi \, \dS \nonumber
	\\
& \quad + \int_\Gamma \left\{-\kappa \Delta_\Gamma\psi   +\varepsilon^{-1}G'(\psi)  \right\} \partial_t\psi \,\dS \nonumber
	\\
& = \int_\Omega\left\{-\varepsilon \Delta \phi  +\varepsilon^{-1}F'(\phi)  \right\} \partial_t\phi  \,\dx 
	\\
& \quad + \int_\Gamma \left\{-\kappa \Delta_\Gamma\psi   +\varepsilon^{-1}G'(\psi) +\varepsilon\partial_n\phi \right\} \partial_t\psi \,\dS \nonumber
	\\
& = \int_\Omega\mu \partial_t\phi \,\dx + \int_\Gamma \nu \partial_t\psi \,\dS \nonumber
	\\
& = \int_\Omega\mu \Delta\mu  \,\dx + \int_\Gamma \nu \Delta_\Gamma\nu  \,\dS \nonumber
	\\
& = -\int_\Omega\nabla\mu\cdot \nabla\mu  \,\dx - \int_\Gamma \nabla_\Gamma \nu \cdot \nabla_\Gamma\nu  \,\dS. \nonumber
	\end{align}
Thus
	\begin{equation}
\diff_t E_{\rm tot} = - \| \nabla \mu \|^2 - \| \nabla_\Gamma \nu \|_\Gamma^2 \le 0 .
    \label{energy dissipation-3} 
	\end{equation}    
See the related references~\cite{Bao2021a, Bao2021b, Knopf2021, LiuC2019, Metzger2021} for more detailed discussions. 

There have been several papers that have analyzed the Cahn-Hilliard equation with Flory-Huggins energy potential~\cite{abels09b, abels07, barrett99, cahn1996, miranville11, elliott92a, debussche95, elliott96b, miranville12, miranville04};  all these works have focused on homogenous Neumann or periodic boundary conditions. If a dynamical boundary condition is involved, the coupled nature makes a theoretical analysis very challenging for the physical system, at both the theoretical and numerical levels. 

Some numerical efforts have been reported for the Cahn-Hilliard equation coupled with dynamical boundary condition, typically with polynomial energy potential. For example, a finite element numerical scheme was proposed in~\cite{Metzger2021}; an energy stability is proved, and a convergence to the weak solution is established. A reaction rate dependent dynamic boundary condition was considered in~\cite{Knopf2021}. Based on the stabilized linearly implicit approach, a first-order-in-time, linear and energy stable scheme was proposed in~\cite{Bao2021b}, and the semi-discrete error analysis was provided as well. A second order stabilized semi-implicit scheme was analyzed in \cite{Meng23}. In addition, a scalar auxiliary variable (SAV) approach was applied to the Cahn-Hilliard-Hele-Shaw system in~\cite{Yao2022}, with the energy stability analysis theoretically justified. 

Meanwhile, a direct extension of these numerical approaches to the Flory-Huggins-Cahn-Hilliard system~\eqref{equation-CHDBC-1} -- \eqref{equation-CHDBC-3}, with dynamical boundary condition, faces many serious difficulties. The singular nature of the logarithmic terms makes the positivity-preserving property (for both $1+ \phi$ and $1- \phi$) very important for the well-defined property of any numerical scheme. A coupled nature between the interior region and the boundary section makes a theoretical analysis even more challenging. In this article, we propose and analyze a numerical scheme for the Flory-Huggins-Cahn-Hilliard system~\eqref{equation-CHDBC-1} -- \eqref{equation-CHDBC-3},  with three theoretical properties justified: positivity-preserving, unique solvability, and unconditional stability for the total energy. 

The numerical approximation to the chemical potential profiles, at both the interior region and on the boundary section, is based on the convex-concave decomposition of the Flory-Huggins energy functional.  An implicit treatment of the nonlinear singular logarithmic term is applied to theoretically justify its positivity-preserving property; 
in fact, the singular and convex nature of the logarithmic term prevents the numerical solution reach the singular limit values~\cite{chen19b}, so that a point-wise positivity is preserved for the logarithmic argument variables. The linear expansive term is explicitly updated to ensure a unique solvability property, due to its concave nature. The surface diffusion term is implicitly treated, which comes from its convexity. Moreover, a numerical approximation to the non-homogeneous Neumann boundary condition for the phase variable turns out to be an essential part of the numerical design. To facilitate the theoretical analysis, we use a standard finite-difference approximation method. The discrete boundary condition for the phase variable, at the next time step, is coupled with the evolutionary equation of the boundary profile. 

The resulting numerical system is highly nonlinear, singular and coupled; both the interior and boundary numerical operators are involved. The unique solvability and positivity-preserving analysis for the proposed numerical scheme turns out to be highly challenging. By a careful convexity analysis over the proposed finite difference approximation, it is discovered that such a numerical system could be represented as a minimization of a discrete numerical energy functional. In more details, this numerical energy contains both the interior and boundary inner products over the associated grid points. Moreover, it is observed that all the singular terms correspond to a discrete convex functional. As a result, a unique solvability and positivity-preserving analysis could be theoretically justified, since the singular nature of the logarithmic terms around the singular limit values prevent the numerical solutions (at both the interior region and on the boundary section) reaching these values. 

The total energy stability of the numerical scheme is a direct consequence of a careful energy estimate, which gives a dissipation law for the discrete version of the total energy. The summation by parts formulas for the physical variables, both at the interior region and on the boundary section, will play an important role in the analysis.

For the rest of this article, the spatial discretization notations are recalled in Section~\ref{sec:numerical scheme}, and the fully discrete finite difference scheme is proposed. The unique solvability and positivity preserving analysis is established in Section~\ref{sec:positivity}.  The total energy stability estimate is provided in Section~\ref{sec:energy stability}. Some numerical results are presented in Section~\ref{sec:numerical results}. Finally, some concluding remarks are made in Section~\ref{sec:conclusion}.

	\section{Numerical scheme}
	\label{sec:numerical scheme}

	\subsection{Finite difference spatial discretization}
	\label{subsec:finite difference}
	
A semi-standard finite difference spatial approximation is applied. We present the numerical approximation on a two-dimensional computational domain $\Omega = (0,1)^2$. An extension to three-dimensional domain will be straightforward, and omitted for brevity. For further simplicity of presentation, we assume a periodic boundary condition in the $x$ direction and physical boundary conditions imposed at the top and bottom of the domain, namely, at 
	\begin{equation} 
\Gamma_{B} := \left\{ (x,y) \ \middle| \  0\le x\le 1, \  y=0\right\}  ,\quad \Gamma_{T}:= \left\{ (x,y) \ \middle| \ 0\le x\le 1,\  y=1 \right\}. 
	\label{boundary-Gamma-1} 
	\end{equation} 
The case of physical boundary conditions on all four boundary sections could be analyzed in a similar manner. In addition, a uniform spatial mesh size, $\Delta x = \Delta y = h = \frac{1}{N}$ with $N \in\mathbb{N}$, is assumed.
We define  
	\[
{\mathcal V}_{\mathrm{p},x}(\Omega) := \left\{ f_{i,j}  \  \middle| \ f_{i,j} = f_{i+\alpha N,j }, \ \forall \, i,\alpha \in \mathbb{Z}, \  j=0,\cdots, N \right\},
	\]
which is the set of grid functions with discrete periodic boundary conditions imposed in the $x$-direction. In particular, the subscripted symbols ${\mathrm{p},x}$ indicate throughout the paper that periodic boundary conditions are imposed in only the $x$-direction.  Herein the notation $f_{i,j}$ represents  the numerical value of $f\in {\mathcal V}_{\mathrm{p},x}(\Omega)$ at the real-space point $( p_i, p_j)\in \mathbb{R}^2$, where $p_i:=i\cdot h$. Analogously, we define, for any $m\in\{0,1,2,\ldots \}$, 
	\[
C_{\mathrm{p},x}^m(\Omega) := \left\{ f\in C^m( \mathbb{R}\times [0,1];\mathbb{R}) \ \middle| \ f(x,y) = f(x+\alpha,y), \ \forall \, \alpha\in\mathbb{Z}, \, \forall \,  x\in\mathbb{R}, \ \forall \, y\in[0,1] \right\}.
	\]
We can naturally define a projection operator $P_h:C_{\mathrm{p},x}^0(\Omega) \to {\mathcal V}_{\mathrm{p},x}(\Omega)$ via
	\[
P_h(f)_{i,j} = f(p_i,p_j), \quad \forall \, i\in\mathbb{Z},\quad  \forall \, j\in \{0, \ldots , N\}.
	\]

To accommodate Neumann boundary conditions in the $y$-direction in our numerical method, we need to add ghost layers at the top and bottom of the domain.	We define grid function space with these ghost layers via
	\[
{\mathcal V}_{\mathrm{p},x}^{+}(\Omega) := \left\{ f_{i,j}  \  \middle| \ f_{i,j} = f_{i+\alpha N,j }, \ \forall \, i,\alpha \in \mathbb{Z}, \  j=-1,\cdots, N+1 \right\}.
	\]
Analogously, we define the grid function spaces
	\[
{\mathcal E}_{\mathrm{p},x}(\Omega) := \left\{ f_{i+\hf,j}  \  \middle| \ f_{i+\hf,j} = f_{i+\hf+\alpha N,j }, \ \forall \, i,\alpha \in \mathbb{Z}, \  j=0,\cdots, N \right\},
	\]
	\[
{\mathcal N}_{\mathrm{p},x}(\Omega) := \left\{ f_{i,j+\hf}  \  \middle| \ f_{i,j+\hf} = f_{i+\alpha N,j +\hf}, \ \forall \, i,\alpha \in \mathbb{Z}, \  j=0,\cdots, N-1 \right\}
	\]
and
	\[
{\mathcal N}_{\mathrm{p},x}^{+}(\Omega) := \left\{ f_{i,j+\hf}  \  \middle| \ f_{i,j+\hf} = f_{i+\alpha N,j +\hf}, \ \forall \, i,\alpha \in \mathbb{Z}, \  j=-1,\cdots, N \right\}.
	\]
We now define the discrete average and difference operators $A_x,D_x: {\mathcal V}_{\mathrm{p},x}(\Omega) \to {\mathcal E}_{\mathrm{p},x}(\Omega)$ via 
	\begin{equation}
 A_x f_{i+\hf,j} := \frac{1}{2}\left(f_{i+1,j} + f_{i,j} \right), \quad D_x f_{i+\hf,j} := \frac{1}{h}\left(f_{i+1,j} - f_{i,j} \right),
 	\end{equation}
for any $f\in {\mathcal V}_{\mathrm{p},x}(\Omega)$. Likewise, we define the discrete average and difference operators $A_y,D_y: {\mathcal V}_{\mathrm{p},x}^{+}(\Omega) \to {\mathcal N}_{\mathrm{p},x}^{+}(\Omega)$ via 
	\begin{equation}
A_y f_{i,j+\hf} := \frac{1}{2}\left(f_{i,j+1} + f_{i,j} \right), \quad D_y f_{i,j+\hf} := \frac{1}{h}\left(f_{i,j+1} - f_{i,j} \right)  .  
	\end{equation} 
for any $f\in {\mathcal V}_{\mathrm{p},x}^{+}(\Omega)$. Analogously, we define the discrete average and difference operators $a_x,d_x: {\mathcal E}_{\mathrm{p},x}(\Omega) \to {\mathcal V}_{\mathrm{p},x}(\Omega)$ via 
	\begin{equation}
a_x f_{i, j} := \frac{1}{2}\left(f_{i+\hf, j} + f_{i-\hf, j} \right),	 \quad d_x f_{i, j} := \frac{1}{h}\left(f_{i+\hf, j} - f_{i-\hf, j} \right),
	\end{equation}
for all $f \in {\mathcal E}_{\mathrm{p},x}(\Omega)$, and we define the discrete average and difference operators $a_y,d_y: {\mathcal N}_{\mathrm{p},x}^{+}(\Omega) \to {\mathcal V}_{\mathrm{p},x}(\Omega)$ via
	\begin{equation}
a_y g_{i,j} := \frac{1}{2}\left(g_{i,j+\hf} + g_{i,j-\hf} \right),	 \quad d_y g_{i,j} := \frac{1}{h}\left(g_{i,j+\hf} - g_{i,j-\hf} \right) ,
	\end{equation}
for all $g \in {\mathcal N}_{\mathrm{p},x}^{+}(\Omega)$. 

For a scalar grid function $g\in {\mathcal V}_{\mathrm{p},x}^{+}(\Omega)$ and a vector function $\vec{f} = ( f^x , f^y)^T$, with $f^x\in {\mathcal E}_{\mathrm{p},x}(\Omega) $ and $f^y\in {\mathcal N}_{\mathrm{p},x}^+(\Omega)$, the discrete divergence is defined as 
	\begin{equation} 
\nabla_h\cdot \big( g \vec{f} \big)_{i,j} = d_x\left( A_x g \, f^x\right)_{i,j}  + d_y\left( A_y g \, f^y\right)_{i,j}  .  
	\label{divergence-1} 
	\end{equation} 
In this case, observe that $\nabla_h\cdot \big( g \vec{f} \big) \in {\mathcal V}_{\mathrm{p},x}(\Omega)$.  Now, suppose that $\phi \in {\mathcal V}_{\mathrm{p},x}^{+}(\Omega)$. Then, naturally, $\nabla_h \phi := ( D_x \phi , D_y \phi)^T$ has $D_x \phi\in {\mathcal E}_{\mathrm{p},x}(\Omega)$ and $D_y \phi\in {\mathcal N}_{\mathrm{p},x}^{+}(\Omega)$. Suppose that $g\in {\mathcal V}_{\mathrm{p},x}^{+}(\Omega)$. Then, we can also define, as above,
	\begin{equation} 
\nabla_h\cdot \big( g \nabla_h \phi \big)_{i,j} =  d_x\left( A_x g \, D_x \phi \right)_{i,j}  + d_y\left( A_y g \, D_y \phi \right)_{i,j} ,
	\label{divergence-2}  
	\end{equation}
where $\nabla_h\cdot \big( g \nabla_h \phi \big) \in {\mathcal V}_{\mathrm{p},x}(\Omega)$. If $g\equiv 1$, then
	\begin{equation}
\Delta_h \phi _{i,j} =  \nabla_h\cdot \big( \nabla_h \phi \big)_{i,j} = d_x\left( D_x \phi \right)_{i,j}  + d_y\left( D_y \phi \right)_{i,j} ,    
	\label{divergence-3}  
	\end{equation} 
where $\Delta \phi \in {\mathcal V}_{\mathrm{p},x}(\Omega)$, which is the usual 5-point stencil.

For two grid functions $f, g \in {\mathcal V}_{\mathrm{p},x}(\Omega)$, the discrete $L^2$ inner product and the associated norm are defined as 
	\begin{equation*}
 \left\langle f , g\right\rangle_\Omega:= h^2 \sum_{i=0}^{N-1} \sum_{j=0}^N \, w_j f_{i,j} g_{i,j} , \quad w_j = \left\{ 
 	\begin{array}{l} 
1 , \, \, \, 1 \le j \le N-1 , 
   \\ 
\frac12 , \, \, \, j=0, N , 
	\end{array} 
\right.  \quad  \| f \|_{2,\Omega} := \sqrt{ \langle f , f \rangle_\Omega\, } . 
	\end{equation*} 
The mean zero space grid function space is defined as 
	\[
\mathring{\mathcal V}_{\mathrm{p},x}(\Omega):=\left\{ f \in {\mathcal V}_{\mathrm{p},x}(\Omega) \ \middle| \ 0 = \overline{f} :=  \frac{1}{| \Omega|} \langle f , {\bf 1} \rangle_\Omega\right\},
	\]
where $|\Omega| = 1$ is the area of $\Omega$. Similarly, for two vector grid functions $\vec{f} = ( f^x , f^y)^T$ and  $\vec{g} = ( g^x , g^y )^T$, with $f^x, g^x \in {\mathcal E}_{\mathrm{p},x}(\Omega)$ and $f^y, g^y\in {\mathcal N}_{\mathrm{p},x}(\Omega)$, the corresponding discrete inner product is defined as 
	\begin{equation*}
\eipvec{\vec{f} }{\vec{g} } : = \eipx{f^x}{g^x}	+ \eipy{g^y}{g^y} ,  \quad \eipx{f^x}{g^x} := \ciptwo{a_x (f^x g^x)}{1} , \quad  \eipy{f^y}{g^y} :=  h^2 \sum_{i,j=0}^{N-1} \, f^y_{i,j+\hf} g^y_{i,j+\hf}  . 
	\end{equation*}	
In addition to the discrete $\| \, \cdot \, \|_{2,\Omega}$ norm, the discrete maximum norm is defined as  
	\[
\nrm{f}_\infty := \max_{\substack{0\le i\le N-1 \\ 0\le j \le N}}\left| f_{i,j}\right|,
	\]
for all $f\in {\mathcal V}_{\mathrm{p},x}(\Omega)$. Discrete $H_h^1$ and $H_h^2$ norms are introduced as for $f \in{\mathcal V}_{\mathrm{p},x}^{+}(\Omega)$,
	\begin{eqnarray*} 
&& \nrm{ \nabla_h f}_2^2 : = \eipvec{\nabla_h f }{ \nabla_h f } = \eipx{D_x f}{D_x f} + \eipy{D_y f}{D_y f} ,
	\\
&& \nrm{f}_{H_h^1}^2 : =  \nrm{f}_2^2+ \nrm{ \nabla_h f}_2^2 ,  \quad 
  \| f \|_{H_h^2}^2 :=  \| f \|_{H_h^1}^2 + \| \Delta_h f \|_{2,\Omega}^2 .  
	\end{eqnarray*}

We need to compute finite differences with respect to the second variable, that is, $y$ or $j$, of grid functions $f\in {\mathcal V}_{\mathrm{p},x}^{+}(\Omega)$, and have these be evaluated at the physical  boundaries $\Gamma_{T}$ and $\Gamma_{B}$. To do this, we define the following operators for any $f\in {\mathcal V}_{\mathrm{p},x}^{+}(\Omega)$
	\[
\tilde{D}_y f_{i,0} := a_y(D_y f)_{i,0} = \frac{f_{i,1}-f_{i,-1}}{2h}	
	\]	
and
	\[
\tilde{D}_y f_{i,N} := a_y(D_y f)_{i,N} = \frac{f_{i,N+1}-f_{i,N-1}}{2h} .	
	\]
Of course, we can extend the definitions to be valid at any grid points, but we only need these differences at the boundary. For instance, we enforce homogeneous Neumann boundary conditions as follows: for any $f\in {\mathcal V}_{\mathrm{p},x}^{+}(\Omega)$
	\[
0 = \tilde{D}_y f_{i,0} = \frac{f_{i,1}-f_{i,-1}}{2h} \quad \iff \quad f_{i, - 1}= f_{i,1}.
	\]
and
	\[
0 =  \tilde{D}_y f_{i,N} = \frac{f_{i,N+1}-f_{i,N-1}}{2h} \quad \iff \quad f_{i,N+1}= f_{i,N-1}.
	\]

On the top and bottom boundary sections $\Gamma_{T}$, $\Gamma_{T}$, the following one-dimensional periodic grid function spaces are utilized
	\[
{\mathcal V}_{\mathrm{p},x}(\Gamma_T) = {\mathcal V}_{\mathrm{p},x}(\Gamma_B) = {\mathcal V}_{\mathrm{p},x}(\Gamma) := \left\{ \varphi_{i}  \ \middle| \ \varphi_{i} = \varphi_{i+\alpha N}, \ \forall \, \alpha,i \in \mathbb{Z} \right\}.	
	\]
The following inner product and norm are introduced: for any $f, g \in {\mathcal V}_{\mathrm{p},x}(\Gamma)$: 
	\begin{equation*}
\left\langle f , g\right\rangle_\Gamma := h \sum_{i=0}^{N-1} f_{i} g_{i} , \quad \nrm{ f }_{2, \Gamma} := \sqrt{ \left\langle f , f\right\rangle_\Gamma}  . 
	\end{equation*} 
We define the operator $\Delta_h^x:{\mathcal V}_{\mathrm{p},x}(\Gamma) \to {\mathcal V}_{\mathrm{p},x}(\Gamma)$ via
	\[
\Delta_h^x f_i := \frac{f_{i+1}-2f_i+f_{i-1}}{h^2} .
	\]
The operators $A_x,D_x$ act in a natural way on the space ${\mathcal V}_{\mathrm{p},x}(\Gamma)$ and we suppress the formulas.

The following summation-by-parts formulae can be proven easily using the techniques found in~\cite{guo16, wang11a, wise10, wise09a} and elsewhere.

	\begin{lem}  
	\label{lemma1} 	
For any $\psi, \phi, g \in {\mathcal V}_{\mathrm{p},x}^{+}(\Omega)$, and any $\vec{f} = (f^x, f^y)^T$, with $f^x\in \mathcal{E}_{\mathrm{p},x}(\Omega)$ and $f^y\in\mathcal{N}_{\mathrm{p},x}^+(\Omega)$, the following summation by parts formulas are valid: 
	\begin{align} 
\ciptwo{\psi}{\nabla_h\cdot\vec{f}} & =  - \eipvec{\nabla_h \psi}{ \vec{f}} + h \sum_{i=0}^{N-1} \Big( \frac12 ( f_{i,N+\hf} + f_{i, N-\hf} ) \psi_{i,N} - \frac12 ( f_{i,\hf} + f_{i,-\hf} ) \psi_{i,0}  \Big) ,  
	\nonumber
	\\
& =  - \eipvec{\nabla_h \psi}{ \vec{f}} + \left\langle a_yf_{\star , N} , \psi_{\star , N} \right\rangle_{\Gamma} - \left\langle a_yf_{\star , 0} , \psi_{\star , 0} \right\rangle_{\Gamma}
	\label{lemma 1-0-1}  
	\\
\ciptwo{\psi}{\nabla_h\cdot \left( g \nabla_h \phi \right)}  & =  -  \eipvec{\nabla_h \psi }{ {\cal A}_h g \nabla_h\phi} + h \sum_{i=0}^{N-1}  \frac12 ( ( A_y g D_y \phi)_{i,N+\hf} + ( A_y g D_y \phi)_{i, N-\hf} ) \psi_{i,N}   
	\nonumber 
	\\
& \quad - h \sum_{i=0}^{N-1}  \frac12 ( ( A_y g D_y \phi)_{i,\hf}  + ( A_y g D_y \phi)_{i,-\hf} ) \psi_{i,0}
	\nonumber
	\\
& =  -  \eipvec{\nabla_h \psi }{ {\cal A}_h g \nabla_h\phi} +  \left\langle a_y(A_yg D_y\phi)_{\star , N},\psi_{\star , N} \right\rangle_\Gamma  -  \left\langle a_y(A_yg D_y\phi)_{\star , 0},\psi_{\star , 0} \right\rangle_\Gamma ,
	\label{lemma 1-0-2}
	\end{align}
where we use the notation
	\[
\eipvec{\nabla_h \psi }{ {\cal A}_h g \nabla_h\phi} := \eipx{D_x\psi}{A_x g D_x\phi} + \eipy{D_y\psi}{A_y gD_y\phi}.
	\]
In particular, if $g \equiv 1$, the following identity is valid: 
	\begin{align} 
\ciptwo{\psi}{\Delta_h \phi} & =  -  \eipvec{\nabla_h \psi }{ \nabla_h\phi} +  \left\langle a_y(D_y\phi)_{\star , N},\psi_{\star , N} \right\rangle_\Gamma  -  \left\langle a_y(D_y\phi)_{\star , 0},\psi_{\star , 0} \right\rangle_\Gamma \nonumber
	\\
& = -  \eipvec{\nabla_h \psi }{ \nabla_h\phi} +  \left\langle \tilde{D}_y\phi_{\star , N},\psi_{\star , N} \right\rangle_\Gamma  -  \left\langle \tilde{D}_y\phi_{\star , 0} , \psi_{\star , 0} \right\rangle_\Gamma .
	\label{lemma 1-0-3} 
	\end{align}
	\end{lem} 
	
We now need to define an important positive, linear operator that will be used in our analysis. First observe the following: for any $\varphi \in\mathring{\mathcal V}_{\mathrm{p},x}(\Omega)$, there is a unique solution $\psi \in \mathring{\mathcal V}_{\mathrm{p},x}^+(\Omega)$ to the problem
	\[
-\Delta_h \psi = \varphi,	
	\]
subject to the boundary conditions
	\[
0 = \tilde{D}_y \psi_{i,0} = \frac{\psi_{i,1}-\psi_{i,-1}}{2h} \quad \iff \quad \psi_{i, - 1}= \psi_{i,1}, \quad 0\le i\le N-1,
	\]
and
	\[
0 =  \tilde{D}_y \psi_{i,N} = \frac{\psi_{i,N+1}-\psi_{i,N-1}}{2h} \quad \iff \quad \psi_{i,N+1}= \psi_{i,N-1}, \quad 0\le i\le N-1.
	\]
The solution operator for this problem defines a one-to-one onto mapping from $\mathring{\mathcal V}_{\mathrm{p},x}(\Omega)$ to $\mathring{\mathcal V}_{\mathrm{p},x}(\Omega)$. We will let $L_h: \mathring{\mathcal V}_{\mathrm{p},x}(\Omega) \to \mathring{\mathcal V}_{\mathrm{p},x}(\Omega)$ represent be the (one-to-one, onto) forward operator in the problem, incorporating the boundary conditions into the definition of the operator. The solution operator is, of course, the inverse, denoted $L_h^{-1}: \mathring{\mathcal V}_{\mathrm{p},x}(\Omega) \to \mathring{\mathcal V}_{\mathrm{p},x}(\Omega)$.

In essence, the forward operator $L_h$ is just negative  the discrete laplacian operator, except near the top and bottom physical boundaries, where the definition of the operator differs from the negative discrete laplacian in order to incorporate the discrete homogeneous Neumann boundary conditions. Specifically, for $\psi\in \mathcal V_{\mathrm{p},x}(\Omega)$, we have
	\begin{equation}
L_h \psi_{i,j} = 
	\begin{cases}
-\Delta_h^x \psi_{i,0} - 2\frac{\psi_{i,1} - \psi_{0,i}}{h^2}, & j = 0,
	\\
-\Delta_h^x \psi_{i,N} - 2\frac{\psi_{i,N-1}-\psi_{N,i}}{h^2}, & j = N,
	\\
-\Delta_h \psi_{i,j}, & \mbox{otherwise}.
	\end{cases}
	\end{equation}
Another interpretation for our definition is that we have removed the need for the ghost layer by incorporating the boundary conditions. We have the following result.
	\begin{prop} 
The operator $L_h: \mathring{\mathcal V}_{\mathrm{p},x}(\Omega) \to \mathring{\mathcal V}_{\mathrm{p},x}(\Omega)$ defined above is positive and symmetric in the sense that
	\[
\cipgen{\psi}{L_h\phi}{\Omega} = \cipgen{L_h\psi}{\phi}{\Omega}, \quad \forall \ \psi,\phi\in \mathring{\mathcal V}_{\mathrm{p},x}(\Omega)
	\]
and
	\[
\cipgen{\psi}{L_h\psi}{\Omega} >0 , \quad \forall \ \psi \in \mathring{\mathcal V}_{\mathrm{p},x}(\Omega), \quad \psi\not\equiv 0.
	\]		
	\end{prop}

For any $\varphi_1 ,\varphi_2,\varphi \in\mathring{\mathcal V}_{\mathrm{p},x}(\Omega)$, we define 
	\begin{equation} 
\langle \varphi_1 , \varphi_2 \rangle_{-1,h,\Omega} :=  \left\langle \varphi_1 , L_h^{-1} (\varphi_2)\right\rangle_\Omega, \quad \| \varphi  \|_{-1,h,\Omega} := \sqrt{\langle \varphi , \varphi \rangle_{-1,h,\Omega}} .
	\end{equation} 
It is straightforward to prove that this defines an inner-product/norm combination.

Likewise, for any $\varphi_1 ,\varphi_2,\varphi \in\mathring{\mathcal V}_{\mathrm{p},x}(\Gamma)$, we define 
	\begin{equation} 
\langle \varphi_1 , \varphi_2 \rangle_{-1,h,\Gamma} =  \left\langle \varphi_1 ,  (-\Delta_h^x)^{-1} (\varphi_2)\right\rangle_\Gamma , \quad 
\| \varphi  \|_{-1,h,\Gamma} = \sqrt{\langle \varphi , \varphi \rangle_{-1,h,\Gamma}} .
	\end{equation} 
As above,  for any $\varphi \in\mathring{\mathcal V}_{\mathrm{p},x}(\Gamma)$, $\psi = (-\Delta^x_h)^{-1}(\varphi) \in \mathring{\mathcal V}_{\mathrm{p},x}(\Gamma)$ is the unique solution to the problem
	\[
-\Delta^x_h \psi = \varphi.
	\]

	\subsection{The fully discrete numerical scheme} 
After taking great pains to define the appropriate operators, we are now in a position to define the finite difference scheme.

The following finite difference scheme is proposed, using the convex-splitting approach: given $\phi^n \in {\mathcal V}_{\mathrm{p},x}(\Omega)$, find  $\phi^{n+1}, \mu^{n+1}\in {\mathcal V}_{\mathrm{p},x}^+(\Omega)$, such that
	\begin{align}  
\frac{\phi^{n+1} - \phi^n}{s } & = \Delta_h \mu^{n+1} ,\label{scheme-CHDBC-1}  
	\\
\mu^{n+1} & = \varepsilon^{-1} ( \ln ( 1+ \phi^{n+1} ) - \ln (1 - \phi^{n+1} ) - \theta_0 \phi^n ) 
     - \varepsilon \Delta_h \phi^{n+1}  ,     
 	\label{scheme-CHDBC-2} 
	\\
\tilde{D}_{y} \mu^{n+1}_{i, 0} & = \tilde{D}_{y} \mu^{n+1}_{i, N} = 0  ,
	\label{scheme-CHDBC-3}
	\\
\phi^{n+1}_{i,0} &= \phi^{n+1}_{B,i} , \quad  \phi^{n+1}_{i,N} = \phi^{n+1}_{T,i}	 , 
 	\label{scheme-CHDBC-4}	
	\\
\frac{\phi^{n+1}_B - \phi^n_B}{s } & = \Delta_h^x \mu_B^{n+1} ,
 	\label{scheme-CHDBC-5}	
	\\
\mu_B^{n+1}  &= \varepsilon^{-1} ( \ln ( 1+ \phi^{n+1}_B ) - \ln (1 - \phi^{n+1}_B ) - \theta_0 \phi^n_B ) - \kappa \Delta_h^x \phi^{n+1}_B  - \varepsilon \tilde{D}_{y} \phi^{n+1}_{\cdot, 0}  ,
	\label{scheme-CHDBC-6}
	\\
\frac{\phi^{n+1}_T - \phi^n_T}{s } & = \Delta_h^x \mu_T^{n+1} ,
	\label{scheme-CHDBC-7}
	\\
\mu_T^{n+1} & = \varepsilon^{-1} ( \ln ( 1+ \phi^{n+1}_T ) - \ln (1 - \phi^{n+1}_T ) - \theta_0 \phi^n_T ) - \kappa \Delta_h^x \phi^{n+1}_T  + \varepsilon \tilde{D}_{y} \phi^{n+1}_{\cdot, N}  .  
	\label{scheme-CHDBC-8}  
	\end{align}
Observe that the ghost points $\phi_{\cdot,-1}$ and $\phi_{\cdot,N+1}$ are involved in the discrete normal derivative term in~\eqref{scheme-CHDBC-6} and ~\eqref{scheme-CHDBC-8}. Further note that second-order spatial accuracy is ensured at the boundary. This present technique requires an assumption that equation~\eqref{equation-CHDBC-1} can be extended beyond the boundary and the exact solution is sufficiently regular, so that the ghost point evaluations (at $j=-1$ and $j=N+1$) could be considered.  By the standard summation-by-parts formulae  introduced in the paper, we can prove easily that the method is mass conservative. In particular, if a solution exists at each time level $n\in\mathbb{N}$, then
	\begin{align}
\overline{\phi^n} & :=\left\langle \phi^n, 1 \right\rangle= \left\langle \phi^0, 1 \right\rangle= \overline{\phi^0} := \beta_0,
	\label{inner_mass}
	\\
\overline{\phi^n_B} & :=\left\langle\phi^n_B, 1 \right\rangle_\Gamma= \left\langle\phi^0_B, 1 \right\rangle_\Gamma = \overline{\phi_B^0}  =: \beta_{B,0},
	\label{bottom_mass}
	\\
\overline{\phi_T^n} & :=\left\langle\phi^n_T, 1 \right\rangle_\Gamma = \left\langle\phi^0_T, 1 \right\rangle_\Gamma =\overline{\phi_T^0} =: \beta_{T,0}.	
	\label{top_mass}
\end{align}

	\section{Unique solvability and positivity preserving properties}  
	\label{sec:positivity} 

We need to define a  subspace of grid functions that are mean-zero in the bulk and boundary regions simultaneously. To do so, we need a special class of grid functions.
	\begin{defn}
The function $f\in\mathcal{V}_{\mathrm{p},x}(\Omega)$ is called a \textbf{bulk-boundary constant-mass function} iff
	\begin{equation}
f_{i,j} := 
	\begin{cases}
f_{B,0}, &  i\in\mathbb{Z}, \quad j = 0, 
	\\
f_{T,0}, &  i\in\mathbb{Z}, \quad j = N,
	\\
f_0,     &  i\in\mathbb{Z}, \quad 1\le j \le N-1, 
	\end{cases}
	\end{equation}
where $f_{B,0}, \, f_{T,0},\, f_0\in\mathbb{R}$ are constants. The vector subspace of all bulk-boundary constant-mass functions is denoted $M$. The grid function subspace
	\[
H:= \left\{  q \in \mathring{\mathcal V}_{\mathrm{p},x}(\Omega) \ \middle| \  q_B:= q_{\star , 0} \in \mathring{\mathcal{V}}_{\mathrm{p},x}(\Gamma_B) \ \mbox{and} \ q_T:= q_{\star , N} \in \mathring{\mathcal{V}}_{\mathrm{p},x}(\Gamma_T)  \right\} 
	\]
is called the \textbf{bulk-boundary mean-zero space}. The grid functions $q_B:= q_{\star , 0} \in \mathring{\mathcal{V}}_{\mathrm{p},x}(\Gamma_B)$ and $q_T:= q_{ \star , N} \in \mathring{\mathcal{V}}_{\mathrm{p},x}(\Gamma_T)$ are called the \textbf{boundary projections} of $q$.
	\end{defn}
	
The following proposition will help with our analysis. Its proof 
	is straightforward and, therefore, omitted.
	\begin{prop}
	\label{prop:bulk-bndry-mean-zero}
Suppose that $\phi\in H$ and $f\in M$, that is, $\phi$ is a bulk-boundary mean-zero grid function and $f$ is a bulk-boundary constant-mass function. Then
	\begin{equation}
\left\langle \phi , f \right\rangle_\Omega = 0,\quad \left\langle \phi_{\star,0} , f_{\star,0} \right\rangle_\Gamma = 0 ,\quad \left\langle \phi_{\star,N} , f_{\star,N} \right\rangle_\Gamma = 0 .
	\end{equation}
Furthermore, suppose $\psi \in  {\mathcal{V}}_{\mathrm{p},x}(\Omega)$ satisfies  $\overline{\psi} = \alpha_0$ (bulk average mass), $\overline{\psi_B}  = \alpha_{B,0}$ (bottom boundary average mass), and $\overline{\psi_T}  = \alpha_{T,0}$ (top boundary average mass), where $\alpha_0, \, \alpha_{B,0}, \,\alpha_{T,0}$ are constants and
	\[
\psi_B : = \psi_{\star,0} \quad \mbox{and}\quad    \psi_T := \psi_{\star, N} 
	\] 
are the boundary projections of $\psi$. Define the bulk-boundary constant mass function
	\begin{equation}
\mathfrak{a}_{i,j} := 
	\begin{cases}
\alpha_{B,0}, & i\in\mathbb{Z}, \quad  j = 0, 
	\\
\alpha_{T,0}, & i\in\mathbb{Z}, \quad  j = N, 
	\\
\frac{1}{1-h} \left[\alpha_0-\frac{h}{2}\left(\alpha_{B,0} + \alpha_{T,0}\right)\right], & i\in\mathbb{Z}, \quad 1\le j \le N-1.
	\end{cases}
	\end{equation}
Then
	\[
\psi - \mathfrak{a} \in H.
	\]
In other words, $\psi-\mathfrak{a}$ is a bulk-boundary mean-zero function.
	\end{prop}

We will need a technical lemma for the proof of Proposition~\ref{prop:equiv-minimization} to follow. The straightforward proof of the lemma is omitted for brevity.
	\begin{lem}
	\label{lem:tech-sumation-H-fun}
Suppose that $f\in \mathcal{V}_{\mathrm{p},x}(\Omega)$,  $g_T\in\mathcal{V}_{\mathrm{p},x}(\Gamma_B)$, and $g_B\in\mathcal{V}_{\mathrm{p},x}(\Gamma_B)$. Suppose that, for all $\psi\in H$, 
	\[
\cipgen{f}{\psi}{\Omega} +\cipgen{g_T}{\psi_T}{\Gamma}+ \cipgen{g_B}{\psi_B}{\Gamma} = 0,
	\]
where we continue to use the convention that $\psi_T = \psi_{\star,N}$ and $\psi_B = \psi_{\star,0}$. Then, there are constants $C_0, C_{T,0}, C_{B,0}\in \mathbb{R}$ such that
	\begin{align}
f_{i,j} + C_0 & = 0, \quad \forall\ i\in\mathbb{N}, \ 1\le j\le N-1,
	\\
\frac{h}{2}f_{T,i} + g_{T,i} + C_{T,0} &= 0 , \quad \forall i\in \mathbb{N},
	\\
\frac{h}{2} f_{B,i} + g_{B,i} + C_{B,0} & = 0 , \quad \forall i\in \mathbb{N}.
	\end{align}
	\end{lem}

To make the expression of the scheme and its properties more compact, let us define
	\[
I(\phi) := (1+\phi)\ln(1+\phi) + (1-\phi)\ln(1-\phi),	
	\]	
from which it follows that
	\[
I'(\phi) = \ln(1+\phi) - \ln(1-\phi).	
	\]
We have the following equivalence property.

	\begin{prop}
	\label{prop:equiv-minimization}
Define the functional
	\begin{equation} 
	\begin{aligned} 
F_h^n (\phi) & :=  \frac{1}{2 s } \|  \phi - \phi^n \|_{-1, h, \Omega}^2 + \frac{1}{h s } \|  \phi_B - \phi^n_B \|_{-1, h, \Gamma}^2 + \frac{1}{h s } \|  \phi_T - \phi^n_T \|_{-1, h, \Gamma}^2
	\\
& \quad + \varepsilon^{-1}\left( \cipgen{I(\phi)}{1}{\Omega}  + \frac{2}{h}   \cipgen{I(\phi_B)}{1}{\Gamma}   + \frac{2}{h} \cipgen{I(\phi_T)}{1}{\Gamma} \right)      
	\\
& \quad  + \frac{\varepsilon}{2} \cipgen{\phi}{L_h\phi}{\Omega}   + \frac{\kappa}{h} \cipgen{\phi_B}{-\Delta_h^x\phi_B}{\Gamma}+\frac{\kappa}{h} \cipgen{\phi_T}{-\Delta_h^x\phi_T}{\Gamma} 
	\\
& \quad - \varepsilon^{-1} \theta_0 \left( \cipgen{\phi^n}{\phi}{\Omega} + \frac{2}{h} \cipgen{\phi_B^n}{\phi_B}{\Gamma} + \frac{2}{h} \cipgen{\phi_T^n}{\phi_T}{\Gamma} \right)  ,  
	\end{aligned} 
    \label{equiv-minimization}
	\end{equation} 
where 
	\[
\phi_B : = \phi_{\star,0}, \quad  \phi_T := \phi_{\star, N},  \quad \phi_B^n : = \phi^n_{\star,0}, \quad   \phi_T^n := \phi^n_{\star, N}
	\] 
are the point-wise boundary projections of $\phi$ and $\phi^n$ into the spaces ${\mathcal V}_{\mathrm{p},x}(\Gamma_B)$ and ${\mathcal V}_{\mathrm{p},x}(\Gamma_T)$, respectively. If $F_h^n$ has a stationary point in the admissible set
	\begin{equation}
A_h :=  \left\{  \phi \in \mathcal{V}_{\mathrm{p},x}(\Omega) \ \middle| \ -1 < \phi_{i,j}  < 1 , \ 0 \le j \le N , \ i\in\mathbb{Z}, \ \overline{\phi} = \beta_0, \ \overline{\phi_B}  = \beta_{B,0}, \ \overline{\phi_T}  = \beta_{T,0} \right\} ,
	\end{equation} 
then this point is a numerical solution of the scheme~\eqref{scheme-CHDBC-1} -- \eqref{scheme-CHDBC-8}. Conversely, if a numerical solution of the scheme~\eqref{scheme-CHDBC-1} -- \eqref{scheme-CHDBC-8} exists, then it is a stationary point of the functional $F_h^n$.
	\end{prop}
	\begin{proof}
To start, observe that the numerical solution of~\eqref{scheme-CHDBC-1} and \eqref{scheme-CHDBC-2} can be equivalently rewritten as 
	\begin{equation}
\varepsilon^{-1} ( I'(\phi^{n+1}) - \theta_0 \phi^n ) - \varepsilon \Delta_h \phi^{n+1} + C_n + \frac{1}{s } L_h^{-1} ( \phi^{n+1} - \phi^n ) = 0 ,
    \label{scheme-CHDBC-rewrite-1}
	\end{equation} 
where $C_n\in\mathbb{R}$ is a constant that must be included since $L_h^{-1} ( \phi^{n+1} - \phi^n )$ is technically mean zero. The boundary conditions for $\phi^{n+1}$ must still applied separately via the use of ghost cells. However, we may effectively eliminate the values at the ghost cell points. To do so, we write~\eqref{scheme-CHDBC-6} and~\eqref{scheme-CHDBC-8} in the equivalent forms
	\begin{align}  
\varepsilon \tilde{D}_{y} \phi^{n+1}_{\star,0} & = \frac{1}{s } ( - \Delta_h^x )^{-1} ( \phi^{n+1}_B - \phi^n_B )+ C_{B,n} + \varepsilon^{-1} ( I'(\phi^{n+1}_B)  - \theta_0 \phi^n_B ) - \kappa \Delta_h^x \phi^{n+1}_B   ,   
	\label{scheme-CHDBC-rewrite-2}	
	\\  
-\varepsilon \tilde{D}_{y} \phi^{n+1}_{\star,N} & = \frac{1}{s} ( - \Delta_h^x )^{-1} ( \phi^{n+1}_T - \phi^n_T )  +C_{T,n} + \varepsilon^{-1} (I'(\phi^{n+1}_T) - \theta_0 \phi^n_T )  - \kappa \Delta_h^x \phi^{n+1}_T , 
	\label{scheme-CHDBC-rewrite-3}	
	\end{align}
where $C_{B,n}$ and $C_{T,n}$ are constants. Now, observe that, in general, for any $\phi\in \mathcal{V}_{\mathrm{p},x}^+(\Omega)$,
	\begin{equation}
-\Delta_h \phi_{i,j} =
	\begin{cases}
L_h \phi_{i,0} +\frac{2}{h}\tilde{D}_y\phi_{i,0}, & \ i\in \mathbb{Z}, \ j = 0, 	
	\\
L_h \phi_{i,N} -\frac{2}{h}\tilde{D}_y\phi_{i,N},& \ i\in \mathbb{Z},	\ j = N, 
	\\
L_h \phi_{i,j} = -\Delta_h\phi_{i,j}, &  \ i\in \mathbb{Z}, \ 1 \le j \le N-1.
	\end{cases}	
	\end{equation}
Thus, at the boundary points, the solution of the scheme must satisfy
	\begin{align}  
- \varepsilon \Delta_h \phi^{n+1}_{i,0} &  = \varepsilon L_h  \phi^{n+1}_{i,0}+ \frac{2 \varepsilon}{h} \tilde{D}_{y} \phi^{n+1}_{i,0}  
	\nonumber 
	\\
& = \varepsilon L_h \phi^{n+1}_{I ,0} + \frac{2}{h} \Big( \frac{1}{s } ( - \Delta_h^x )^{-1} ( \phi^{n+1}_B - \phi^n_B )_i + C_{B,n} 
	\nonumber
	\\
& \quad +\varepsilon^{-1} ( I'(\phi^{n+1}_{B,i} ) - \theta_0 \phi^n_{B,i} )  - \kappa \Delta_h^x \phi^{n+1}_{B,i} \Big) ,
	\label{scheme-CHDBC-rewrite-4}
	\end{align}   
and
	\begin{align}  
 - \varepsilon \Delta_h \phi^{n+1}_{i,N} &  = \varepsilon L_h  \phi^{n+1}_{i,N} - \frac{2 \varepsilon}{h} \tilde{D}_{y} \phi^{n+1}_{i,N}  
	\nonumber 
	\\
& = \varepsilon L_h \phi^{n+1}_{i,N} + \frac{2}{h} \Big( \frac{1}{s } ( - \Delta_h^x )^{-1} ( \phi^{n+1}_T - \phi^n_T )_i + C_{T,n} 
	\nonumber
	\\
& \quad + \varepsilon^{-1} ( I'(\phi^{n+1}_{T,i}) - \theta_0 \phi^n_T ) - \kappa \Delta_h^x \phi^{n+1}_{T,i}  \Big).
	\label{scheme-CHDBC-rewrite-5}
	\end{align}
For the boundary points, then, combining \eqref{scheme-CHDBC-rewrite-1}, \eqref{scheme-CHDBC-rewrite-4}, and \eqref{scheme-CHDBC-rewrite-5}, we have 
	\begin{align}  
0 & = \varepsilon^{-1} ( I'(\phi^{n+1}_{i,0}) - \theta_0 \phi^n_{i,0} ) + C_n + \frac{1}{s} L_h^{-1} (\phi^{n+1} - \phi^n)_{i,0}   
	\nonumber 
	\\
& \quad + \varepsilon L_h \phi^{n+1}_{I ,0} + \frac{2}{h} \Big( \frac{1}{s } ( - \Delta_h^x )^{-1} ( \phi^{n+1}_B - \phi^n_B )_i + C_{B,n} 
	\nonumber
	\\
& \quad +\varepsilon^{-1} ( I'(\phi^{n+1}_{B,i} ) - \theta_0 \phi^n_{B,i} )  - \kappa \Delta_h^x \phi^{n+1}_{B,i} \Big) ,
	\label{scheme-CHDBC-rewrite-6}
	\end{align}   
and
	\begin{align}  
0 & =  \varepsilon^{-1} ( I'(\phi^{n+1}_{i,N}) - \theta_0 \phi^n_{i,N} ) + C_n + \frac{1}{s} L_h^{-1} (\phi^{n+1} - \phi^n)_{i,N} 
	\nonumber 
	\\
& \quad +  \varepsilon L_h \phi^{n+1}_{i,N} + \frac{2}{h} \Big( \frac{1}{s } ( - \Delta_h^x )^{-1} ( \phi^{n+1}_T - \phi^n_T )_i + C_{T,n} 
	\nonumber
	\\
& \quad + \varepsilon^{-1} ( I'(\phi^{n+1}_{T,i}) - \theta_0 \phi^n_T ) - \kappa \Delta_h^x \phi^{n+1}_{T,i}  \Big).
	\label{scheme-CHDBC-rewrite-7}
	\end{align}
	
\noindent{($\implies$):} Now, suppose that $F_h^n$ has a stationary point, $\phi\in A_h$. In other words, for any $\psi\in H$, 
	\[
0 = \left.\frac{\diff F_h^n}{\diff\tau}\left(\phi+\tau \psi\right)\right|_{\tau = 0}	.
	\]
First, observe that, given $\phi\in A_h$ and $\psi\in H$,  $\phi+\tau \psi\in A_h$, provided $\tau\in\mathbb{R}$ is sufficiently small. It follows that  
	\begin{align}
0 & =  \left.\frac{\diff F_h^n}{\diff\tau}\left(\phi+\tau \psi\right)\right|_{\tau = 0}
	\nonumber
	\\
& =  \frac{1}{s} \cipgen{\psi}{L_h^{-1}(\phi - \phi^n)}{\Omega} + \frac{2}{h s} \cipgen{\psi_B}{(-\Delta_h^x)^{-1}(\phi_B - \phi^n_B)}{\Gamma}  + \frac{2}{h s} \cipgen{\psi_T}{(-\Delta^x_h)^{-1}(\phi_T - \phi^n_T)}{\Gamma}
	\nonumber
	\\
& \quad + \varepsilon^{-1}\left( \cipgen{I'(\phi)}{\psi}{\Omega}  + \frac{2}{h}   \cipgen{I'(\phi_B)}{\psi_B}{\Gamma}   + \frac{2}{h} \cipgen{I'(\phi_T)}{\psi_T}{\Gamma} \right)
	\nonumber   
	\\
& \quad + \varepsilon  \cipgen{\psi}{L_h\phi}{\Omega}   + \frac{\kappa}{h} \cipgen{\psi_B}{-\Delta_h^x\phi_B}{\Gamma}+\frac{\kappa}{h} \cipgen{\psi_T}{-\Delta_h^x\phi_T}{\Gamma} 
	\nonumber
	\\
& \quad - \varepsilon^{-1} \theta_0 \left( \cipgen{\phi^n}{\psi}{\Omega} + \frac{2}{h} \cipgen{\phi_B^n}{\psi_B}{\Gamma} + \frac{2}{h} \cipgen{\phi_T^n}{\psi_T}{\Gamma} \right) .  
	\end{align}
Appealing to our technical lemma, Lemma~\ref{lem:tech-sumation-H-fun}, it follows that
	\begin{equation}
\frac{1}{s} L_h^{-1}(\phi - \phi^n)_{i,j} + \varepsilon^{-1}I'(\phi_{i,j}) + \varepsilon L_h\phi_{i,j} - \varepsilon^{-1} \theta_0  \phi^n_{i,j} + \tilde{C}_n = 0, \quad \forall \, i\in\mathbb{N}, \ 1\le j \le N-1,
	\label{scheme-CHDBC-rewrite-1-alt}
	\end{equation}
where $\tilde{C}_n\in\mathbb{R}$ is a constant. Furthermore, at the top boundary points,
	\begin{align}
0 &= \frac{\varepsilon^{-1}}{2} ( I'(\phi^{n+1}_{i,N}) - \theta_0 \phi^n_{i,N} ) + \frac{\varepsilon}{2} L_h \phi^{n+1}_{i,N} + \frac{1}{2 s } L_h^{-1} ( \phi^{n+1} - \phi^n )_{i,N} &
	\nonumber
	\\
& \quad +  \frac{1}{h s} (-\Delta^x_h)^{-1}(\phi_T - \phi^n_T)_i + \frac{1}{h}\varepsilon^{-1}\left(I'(\phi_{T,i}) - \theta_0\phi^n_{T,i} \right) -\frac{1}{h}\kappa\Delta_h^x\phi_{T,i} + \tilde{C}_{T,n},
	\label{scheme-CHDBC-rewrite-6-alt}
	\end{align}
where $\tilde{C}_{T,n}\in\mathbb{R}$ is a constant, and at the bottom boundary points,
	\begin{align}
0 &= \frac{\varepsilon^{-1}}{2} ( I'(\phi^{n+1}_{i,0}) - \theta_0 \phi^n_{i,0} ) + \frac{\varepsilon}{2} L_h \phi^{n+1}_{i,0} + \frac{1}{2 s } L_h^{-1} ( \phi^{n+1} - \phi^n )_{i,0} &
	\nonumber
	\\
& \quad +  \frac{2}{h s} (-\Delta^x_h)^{-1}(\phi_B - \phi^n_B)_i + \frac{1}{h}\varepsilon^{-1}\left(I'(\phi_{B,i}) - \theta_0\phi^n_{B,i} \right) -\frac{1}{h}\kappa\Delta_h^x\phi_{B,i} + \tilde{C}_{B,n},
	\label{scheme-CHDBC-rewrite-7-alt}
	\end{align}
where $\tilde{C}_{B,n}\in\mathbb{R}$ is a constant. Comparing \eqref{scheme-CHDBC-rewrite-1}, \eqref{scheme-CHDBC-rewrite-6}, and \eqref{scheme-CHDBC-rewrite-7} with \eqref{scheme-CHDBC-rewrite-1-alt}, \eqref{scheme-CHDBC-rewrite-6-alt}, and \eqref{scheme-CHDBC-rewrite-7-alt}, respectively, and adjusting the constants as necessary, we have the result.

	\medskip

\noindent{($\impliedby$):} The other direction is simple. If \eqref{scheme-CHDBC-rewrite-1-alt}, \eqref{scheme-CHDBC-rewrite-6-alt}, and \eqref{scheme-CHDBC-rewrite-7-alt} all hold, then the solution of those three equations must be a stationary point of $F_h^n$.
	\end{proof}
	
	\begin{prop}
	\label{prop:convexity}
The functional $F_h^n$ is strictly convex over the admissible set $A_h$.
	\end{prop}
	\begin{proof}
Fix $\phi\in A_h$. Let $\psi\in H$ be arbitrary. For all $\tau$ in a (sufficiently small) neighborhood of 0, $\phi + \tau\psi\in A_h$. The second variation is
	\begin{align}
\left.\frac{\diff^2 F_h^n}{\diff \tau^2}(\phi+\tau\psi) \right|_{\tau = 0} & =  \frac{1}{s} \cipgen{\psi}{L_h^{-1}\psi}{\Omega} + \frac{2}{h s} \cipgen{\psi_B}{(-\Delta_h^x)^{-1}\psi_B}{\Gamma}  + \frac{2}{h s} \cipgen{\psi_T}{(-\Delta^x_h)^{-1}\psi_T}{\Gamma}
	\nonumber
	\\
& \quad + \varepsilon^{-1}\left( \cipgen{\psi I''(\phi)}{\psi}{\Omega}  + \frac{2}{h}   \cipgen{\psi_B I''(\phi_B)}{\psi_B}{\Gamma}   + \frac{2}{h} \cipgen{\psi_T I''(\phi_T)}{\psi_T}{\Gamma} \right)
	\nonumber   
	\\
& \quad + \varepsilon  \cipgen{\psi}{L_h\psi}{\Omega} + \frac{\kappa}{h} \cipgen{\psi_B}{-\Delta_h^x\psi_B}{\Gamma}+\frac{\kappa}{h} \cipgen{\psi_T}{-\Delta_h^x\psi_T}{\Gamma} 
	\nonumber
	\\
& \ge 0 .  
	\end{align} 
The variation is strictly positive if $\psi\in H$ is not identically the zero grid function.
	\end{proof}
	
Our next task is to prove that $F_h^n$ always has a unique minimizer, which is also a stationary point. The following preliminary estimates are needed for the unique solvability analysis. The proofs are similar to an associated result in~\cite{chen19b} and are omitted for brevity.
	\begin{lem}[\cite{chen19b}]
	\label{CHDBC-positivity-Lem-0}  
Suppose that $\varphi^\star$, $\hat{\varphi} \in \mathcal{V}_{\mathrm{p},x}(\Omega)$, with $\hat{\varphi} - \varphi^\star\in\mathring{\mathcal{V}}_{\mathrm{p},x}(\Omega)$. Assume that $-1 < \hat{\varphi}_{i,j} , \varphi^\star_{i,j} < 1$, for all $0 \le i,j \le N$. The following estimate is valid: 
	\begin{equation} 
\|  L_h^{-1} ( \hat{\varphi} - \varphi^\star ) \|_{\infty,\Omega} \le C_1 ,
  	\label{CHDBC-Lem-0} 
	\end{equation} 
where $C_1>0$ only depends on $\Omega$. Similarly, for $f, g \in \mathcal{V}_{\mathrm{p},x}(\Gamma)$, with $f - g \in\mathring{\mathcal{V}}_{\mathrm{p},x}(\Gamma)$, and $-1 < f_{i} , g_{i} < 1$, for all $0 \le i \le N$, the following estimate becomes available: 
	\begin{equation} 
\| (-\Delta_h^x)^{-1} ( f-g ) \|_{\infty, \Gamma} \le C_2 ,
  	\label{CHDBC-Lem-1} 
	\end{equation} 
where $C_2>0$ only depends on $\Gamma_T$, $\Gamma_B$. 
	\end{lem}

	\begin{thm}  
	\label{CHDBC-positivity} 
Given $\phi^n \in  {\mathcal{V}}_{\mathrm{p},x}(\Omega)$, with $-1 < \phi^n_{i,j} < 1$, $0 \le i, j \le N$,  and $\overline{\phi^n} = \beta_0$, $\overline{\phi^n_B}  = \beta_{B,0}$, $\overline{\phi^n_T}  = \beta_{T,0}$, there exists a unique solution $\phi^{n+1} \in {\mathcal{V}}_{\mathrm{p},x}(\Omega)$ to the scheme \eqref{scheme-CHDBC-1} -- \eqref{scheme-CHDBC-8}, with $-1 < \phi^{n+1}_{i,j} < 1$, $0 \le i, j \le N,$ and $\overline{\phi^{n+1}} = \beta_0$,  $\overline{\phi^{n+1}_B}  = \beta_{B,0}$, $\overline{\phi^{n+1}_T}  = \beta_{T,0}$. In particular, the solution is the unique minimizer and stationary point of the functional $F_h^n$, and we write
	\[
\phi^{n+1} = \mathop{\mathrm{argmin}}_{\phi\in A_h} F_h^n(\phi).
	\]
	\end{thm} 

	\begin{proof} 
Recall that $F^n_h $ is a strictly convex functional over the domain $A_h$. Consider the following closed, convex domain: for $\delta \in (0,1/2)$, 
	\begin{equation}
A_{h,\delta} :=  \left\{ \phi \in A_h  \ \middle| \ -1 + \delta \le \phi_{i,j}   \le 1-\delta, \ 0 \le j \le N, \ i\in\mathbb{Z}  \right\}  .
	\end{equation} 
Since $A_{h,\delta}$ is a convex, compact set, there exists a (not necessarily unique) minimizer of $F^n_h$ over $A_{h,\delta}$. The key point of our positivity analysis is that such a minimizer could not occur at one of the boundary points of $A_{h,\delta}$, if $\delta$ is sufficiently small. 

Fix $\delta\in (0,1/2)$. Let us suppose that the minimizer $\phi^\star\in A_{h,\delta}$ of $F_h^n$ occurs at a boundary point of $A_{h,\delta}$, by which we mean that $\phi^\star\in A_h$ and 
	\[
\nrm{\phi^\star}_{\infty,\Omega} = 1-\delta.
	\]
Without loss of generality, we assume that
	\[
\phi^\star_{{i_0},{j_0}} = -1 + \delta,
	\]
for some grid point $({i_0}, {j_0})$. Suppose that $\phi^\star$ attains its maximum value at the point $({i_1}, {j_1})$. To start with, let us assume that neither $(i_0, j_0)$ nor $(i_1, j_1)$ appear on one of the physical boundary sections, $\Gamma_{B}$ or $\Gamma_{T}$. The physical boundary cases will be analyzed later. By the fact that $\overline{\phi^\star} = \beta_0$, it is obvious that $\phi^\star_{{i_1}, {j_1}} \ge \beta_0$. 

Consider the directional derivative: for any $\psi\in H$, 
	\begin{align}
& \hspace{-0.2in} \left.\frac{\diff F_h^n}{\diff\tau}\left(\phi+\tau \psi\right)\right|_{\tau = 0}  
	\nonumber
	\\
& =   \frac{1}{s} \cipgen{\psi}{L_h^{-1}(\phi - \phi^n)}{\Omega} + \frac{2}{h s} \cipgen{\psi_B}{(-\Delta_h^x)^{-1}(\phi_B - \phi^n_B)}{\Gamma}  + \frac{2}{h s} \cipgen{\psi_T}{(-\Delta^x_h)^{-1}(\phi_T - \phi^n_T)}{\Gamma}
	\nonumber
	\\
& \quad + \varepsilon^{-1}\left( \cipgen{I'(\phi)}{\psi}{\Omega}  + \frac{2}{h}   \cipgen{I'(\phi_B)}{\psi_B}{\Gamma}   + \frac{2}{h} \cipgen{I'(\phi_T)}{\psi_T}{\Gamma} \right)
	\nonumber   
	\\
& \quad + \varepsilon  \cipgen{\psi}{L_h\phi}{\Omega}   + \frac{\kappa}{h} \cipgen{\psi_B}{-\Delta_h^x\phi_B}{\Gamma}+\frac{\kappa}{h} \cipgen{\psi_T}{-\Delta_h^x\phi_T}{\Gamma} 
	\nonumber
	\\
& \quad - \varepsilon^{-1} \theta_0 \left( \cipgen{\phi^n}{\psi}{\Omega} + \frac{2}{h} \cipgen{\phi_B^n}{\psi_B}{\Gamma} + \frac{2}{h} \cipgen{\phi_T^n}{\psi_T}{\Gamma} \right) . 
	\label{CHDBC-positive-2}
	\end{align}	
Let us pick the direction $\psi \in H$, such that  
	\begin{equation} 
\psi_{i,j} = \delta_{i,i_0}\delta_{j,j_0} - \delta_{i,i_1}\delta_{j,j_1} , 
	\label{CHDBC-positive-2-2}
	\end{equation}  
where $\delta_{k,\ell}$ is the Kronecker delta function. The following equality is valid: 
	\begin{align}
\frac{1}{h^2}\diff_\tau \! \left.F^n_h(\phi^\star +\tau\psi ) \right|_{\tau=0} & = \frac{1}{s } \Big( L_h^{-1} ( \phi^\star-\phi^n )_{i_0,j_0} - L_h^{-1} ( \phi^\star-\phi^n )_{i_1,j_1} \Big)  
   	\nonumber 
	\\
& \quad  + \varepsilon^{-1}  \Big(  \ln ( 1+ \phi^\star_{i_0,j_0}) 
   - \ln ( 1- \phi^\star_{i_0,j_0} )  \Big) 
	\nonumber 
	\\
& \quad  + \varepsilon^{-1}  \Big( - \ln ( 1+ \phi^\star_{i_1,j_1} ) + \ln ( 1- \phi^\star_{i_1,j_1} )   \Big)   
	\nonumber 
	\\ 	
& \quad + \varepsilon ( L_h \phi^\star_{i_0,j_0}  - L_h \phi^\star_{i_1,j_1}  ) - \varepsilon^{-1} \theta_0 ( \phi^n_{i_0,j_0} - \phi^n_{i_1,j_1} )  .
	\label{CHDBC-positive-3} 
	\end{align}
Because of the fact that $\phi^\star_{{i_0}, {j_0}} = -1 + \delta$ and $\phi^\star_{{i_1},{j_1}}  \ge \beta_0$, we have  
	\begin{align*} 
\ln \left( 1+ \phi^\star_{i_0,j_0} \right)  -  \ln \left( 1 - \phi^\star_{i_0,j_0}\right) & = \ln (\delta) - \ln ( 2 - \delta_0 ) = \ln\left(\frac{\delta}{2-\delta}\right) , 
	\\
- \ln \left( 1- \phi^\star_{i_1,j_1} \right)  +  \ln \left( 1- \phi^\star_{i_1,j_1} \right) & \le - \ln (1+\beta_0) + \ln ( 1 - \beta_0 ) = \ln \left(  \frac{1 - \beta_0}{1+\beta_0} \right) . 
	\end{align*} 
Thus,
	\begin{align}
& \hspace{-0.75in} \varepsilon^{-1}\left(\ln \left( 1+ \phi^\star_{i_0,j_0} \right)  -  \ln \left( 1 - \phi^\star_{i_0,j_0}\right) - \ln \left( 1- \phi^\star_{i_1,j_1} \right)  +  \ln \left( 1- \phi^\star_{i_1,j_1} \right)\right)
	\nonumber
	\\
& \le \varepsilon^{-1} \ln\left(\frac{\delta}{2-\delta}\right) + \varepsilon^{-1} \ln \left(  \frac{1 - \beta_0}{1+\beta_0} \right).
	\label{CHDBC-positive-4} 
	\end{align}
For the first two terms appearing in~\eqref{CHDBC-positive-3}, an application of Lemma~\ref{CHDBC-positivity-Lem-0} indicates that 
	\begin{equation} 
- 2 C_1  \le L_h^{-1} ( \phi^\star-\phi^n )_{i_0,j_0} 
   - L_h^{-1} ( \phi^\star-\phi^n )_{i_1,j_1}  \le  2 C_1 .  
	\label{CHDBC-positive-5} 
	\end{equation}
Since $\phi^\star$ achieves its minimum value at the grid point $(i_0, j_0)$, with 
	\[
-1 + \delta  = \phi^\star_{i_0, j_0} \le \phi^\star_{i,j}
	\]
for any arbitrary grid point $(i,j)$, and its maximum value at the grid point $(i_1, j_1)$, with 
	\[
\phi^\star_{i,j} \le \phi^\star_{i_1, j_1} \le 1- \delta, 
	\]
for any $(i,j)$, we conclude that 
	\begin{equation*} 
L_h \phi^\star_{i_0,j_0} = - \Delta_h \phi^\star_{i_0,j_0} \le 0   \quad \mbox{and} \quad  L_h\phi^\star_{i_1,j_1} = -\Delta_h \phi^\star_{i_1,j_1} \ge 0 ,
	\end{equation*} 
which implies that 
	\begin{equation}
\varepsilon\left(L_h \phi^\star_{i_0,j_0} -  L_h\phi^\star_{i_1,j_1}\right) \le 0.
	\label{CHDBC-positive-6} 
	\end{equation}
For the numerical solution $\phi^n$ at the previous time step, a point-wise bound $\nrm{\phi^n}_\infty \le  1$ indicates that 
	\begin{equation} 
-\frac{2\theta}{\varepsilon} \le  \phi^n_{i_0, j_0} - \phi^n_{i_1, j_1} \le \frac{2\theta}{\varepsilon} .
	\label{CHDBC-positive-7} 
	\end{equation}
Combining substitution of~\eqref{CHDBC-positive-4} -- \eqref{CHDBC-positive-7} into~\eqref{CHDBC-positive-3} yields 
	\begin{equation} 
\frac{1}{h^2}\diff_\tau \! \left.F_h^n(\phi^\star +\tau\psi )\right|_{\tau=0} \le \frac{1}{\varepsilon} \left(\ln\left(\frac{\delta}{2-\delta}\right) +  \ln \left(  \frac{1 - \beta_0}{1+\beta_0} \right)\right) + 2 \frac{C_1}{ s}  + 2 \frac{\theta_0}{\varepsilon}  .  
	\label{CHDBC-positive-8} 
	\end{equation}  
Define the following quantity: 
	\begin{equation} 
C_3 := \frac{1}{\varepsilon}  \ln \left(  \frac{1 - \beta_0}{1+\beta_0} \right) + 2 \frac{C_1}{ s}  + 2 \frac{\theta_0}{\varepsilon}  .
	\label{eqn:C_3} 
\end{equation} 
It is obvious that $C_3$ is a fixed, finite number, provided the time step size $s>0$ and the interface parameter $\varepsilon>0$ are fixed, though the parameter blows up to infinity, as $s,\varepsilon \searrow 0$. For fixed $s$ and $\varepsilon$, we may choose $1/2> \delta>0$ small enough so that 
	\begin{equation} 
\frac{1}{\varepsilon} \ln\left( \frac{\delta}{2+\delta}\right)  + C_3 < 0 .  
	\label{CHDBC-positive-9} 
	\end{equation} 
This in turn reveals that 
	\begin{equation} 
\diff_\tau \! \left.F^n_h(\phi^\star +\tau\psi )\right|_{\tau=0}  \le \frac{1}{\varepsilon} \ln\left( \frac{\delta}{2+\delta}\right)  + C_3 < 0< 0 .  
	\label{CHDBC-positive-10} 
	\end{equation} 
If $\delta$ is sufficiently small, such an inequality contradicts the assumption that $F^n_h$ has a minimum at $\phi^\star$, since the directional derivative is negative in a direction pointing into the interior of $A_{h,\delta}$. 

In the case that the minimum is achieved at one of the physical boundaries, the argument is a little more involved. Let us  assume the minimizer is given by $\phi^\star$, with $\phi^\star_{{i_0},0} = -1 + \delta$, at the physical boundary grid point $({i_0}, 0) \in \Gamma_{B}$. Meanwhile, we suppose that $\phi^\star$ attains its maximum value over the bottom boundary section, at the boundary point $({i_1}, 0)$. As usual, we denote by $\phi^\star_B$, $\phi^\star_T$  the point-wise projections of $\phi^\star$ onto $\mathcal{V}_{\mathrm{p},x}(\Gamma_{B})$ and $\mathcal{V}_{\mathrm{p},x}(\Gamma_{T})$, respectively. Because of the fact that $\overline{\phi^\star_B}  = \beta_{B,0}$, it is obvious that $\phi^\star_{{i_1}, 0} \ge \beta_{B,0}$. 

We use the direction (test) function $\psi \in H$ again as defined  in~\eqref{CHDBC-positive-2-2}, but with $j_0 = j_1 =0$. Subsequently, the directional derivative becomes 
	\begin{align}
\frac{1}{h^2}\diff_\tau \! \left.F^n_h(\phi^\star +\tau\psi ) \right|_{\tau=0} & = \frac{1}{2s} \Big( L_h^{-1} ( \phi^\star-\phi^n )_{i_0,0} -  L_h^{-1} ( \phi^\star-\phi^n )_{i_1,0} \Big)  
   	\nonumber 
	\\
& \quad + \frac{2}{h s } \Big(  (- \Delta_h^x)^{-1} ( \phi^\star_B - \phi^n_B )_{i_0}  - (- \Delta_h^x)^{-1} ( \phi^\star_B - \phi^n_B )_{i_1} \Big)     
	\nonumber 	
	\\
& \quad  + \varepsilon^{-1}  ( \frac12 + 2h^{-1} ) \Big(  \ln ( 1+ \phi^\star_{i_0,0}) - \ln ( 1- \phi^\star_{i_0,0}  )  \Big) 
	\nonumber 
	\\
& \quad + \varepsilon^{-1}  ( \frac12+  2h^{-1} ) \Big( - \ln ( 1+ \phi^\star_{i_1,0}  ) + \ln ( 1- \phi^\star_{i_1,0}  ) \Big)
	\nonumber 
	\\ 	
& \quad + \frac{\varepsilon}{2} ( L_h \phi^\star_{i_0,0}  - L_h \phi^\star_{i_1,0}  ) - \kappa h^{-1} ( \Delta_h^x \phi^\star_{i_0}  - \Delta_h^x \phi^\star_{i_1} )    
	\nonumber 
	\\
& \quad - \varepsilon^{-1} \theta_0 ( \frac12 + 2h^{-1} ) ( \phi^n_{i_0,0} - \phi^n_{i_1,0} )  .
	\label{CHDBC-positive-11-1} 
	\end{align}
In comparison with~\eqref{CHDBC-positive-3}, in which the minimum value point $(i_0, j_0)$ is located within the interior of the grid, a distinguishing feature of the directional derivative~\eqref{CHDBC-positive-11-1} is the coefficient in front of the singular logarithmic terms, which is $\varepsilon^{-1}  ( \frac12+ 2h^{-1} )$, instead of $\varepsilon^{-1}$ in~\eqref{CHDBC-positive-3}. In either case, the corresponding coefficients are positive, and this crucial fact will play an essential role in the positivity-preserving analysis. 

Similar to the inequalities~\eqref{CHDBC-positive-4} -- \eqref{CHDBC-positive-7}, the following estimates are valid:  
	\begin{align} 
& \hspace{-0.5in} \frac{1/2 + 2/h}{\varepsilon}\left( \ln ( 1+ \phi^\star_{i_0,0} )  -  \ln ( 1+ \phi^\star_{i_0,0}  ) - \ln ( 1- \phi^\star_{i_0,0}  )  +  \ln ( 1- \phi^\star_{i_1,0}  )\right) 
	\\
& \le  \frac{1/2 + 2/h}{\varepsilon}\left(  \ln\left(\frac{\delta}{2-\delta}\right) +  \ln \left(  \frac{1 - \beta_{B,0}}{1+\beta_{B,0}} \right)\right)
	\label{CHDBC-positive-12-2}
	\end{align}
	\begin{equation}
-  \frac{C_1}{s}  \le \frac{1}{2 s}\left(  L_h^{-1} ( \phi^\star-\phi^n )_{i_0,0} -  L_h^{-1} ( \phi^\star-\phi^n )_{i_1,0} \right) \le  \frac{C_1}{s} ,   
	\label{CHDBC-positive-12-3} 
	\end{equation}
	\begin{equation}
- 2 \frac{C_2}{hs}  \le \frac{1}{hs} \left( (- \Delta_h^x)^{-1} ( \phi^\star_B - \phi^n_B )_{i_0} - (- \Delta_h^x)^{-1} ( \phi^\star_B - \phi^n_B )_{i_1} \right) \le  2 \frac{C_2}{hs} ,   
	\label{CHDBC-positive-12-4} 
	\end{equation}	
	\begin{equation}
\frac{\varepsilon}{2} \left( L_h\phi^\star_{i_0,0} -L_h \phi^\star_{i_1,0} \right) \le 2 h^{-2} ,
	\label{CHDBC-positive-12-5} 
	\end{equation}
	\begin{equation} 
 \frac{\kappa}{h} \left(- \Delta_h^x \phi^\star_{i_0}  + \Delta_h^x \phi^\star_{i_1} \right) \le 0 , 
	\label{CHDBC-positive-12-6} 	
	\end{equation}
and, finally, 
	\begin{equation}
- \frac{\theta_0 ( 1 + 4/h)}{\varepsilon} \le  \frac{\theta_0 ( 1/2 + 2/h)}{\varepsilon}\left( \phi^n_{i_0,0} - \phi^n_{i_1,0} \right) \le  \frac{\theta_0 ( 1 + 4/h)}{\varepsilon}.
	\label{CHDBC-positive-12-7} 
	\end{equation}
Subsequently,  a substitution of~\eqref{CHDBC-positive-12-2} -- \eqref{CHDBC-positive-12-7} into~\eqref{CHDBC-positive-11-1} results in
	\begin{equation} 
\frac{1}{h^2}\diff_\tau \! \left.F_h^n(\phi^\star +\tau\psi )\right|_{\tau=0} \le \frac{1/2+2/h}{\varepsilon} \ln\left( \frac{\delta}{2+\delta}\right)  + C_4 , 
	\label{CHDBC-positive-13} 
	\end{equation}
where
	\[
C_4 :=  \frac{ 1/2 + 2/h }{\varepsilon}  \ln \left( \frac{1- \beta_{B,0}}{1+\beta_{B,0}}  \right) +  \frac{ C_1 + 4  C_2/h }{s}  + \frac{2 \varepsilon}{h^2} +  \theta_0 \frac{ 1 + 4/h }{\varepsilon}   .  
	\] 
As before, $C_4$  a constant for fixed $s >0$, $h>0$, and $\varepsilon>0$, though it is  singular, as $s, h, \varepsilon  \searrow 0$. Of course, for any fixed $s$, $\varepsilon$, and $h$, the value of $\delta>0$ could be chosen to be sufficiently small so that 
	\begin{equation} 
\frac{1}{h^2}\diff_\tau \! \left.F_h^n(\phi^\star +\tau\psi )\right|_{\tau=0} \le \frac{1/2+2/h}{\varepsilon} \ln\left( \frac{\delta}{2+\delta}\right)  + C_4 < 0.
	\label{CHDBC-positive-14} 
	\end{equation}
Again, this inequality contradicts the assumption that $F^n_h$ has a minimum at $\phi^\star$, since the directional derivative is negative in a direction pointing into the interior of $A_{h,\delta}$. 

Using a similar analysis, we are able to prove that, if the minimizer is given by $\phi^\star$, with $\phi^\star_{{i_0},N}=-1 + \delta$, at a boundary grid point $({i_0}, N) \in \Gamma_{T}$, the same inequality~\eqref{CHDBC-positive-10} becomes available, which in turn leads to a contradiction. 

One more case needs to be taken into consideration. If the minimum value of the minimizer function occurs at one interior point, with $\phi^\star_{i_0, j_0} = -1 + \delta$, while its maximum value is achieved at one boundary point, say $(i_1, 0)$, without loss of generality. Again, an obvious bound becomes available, $\phi^\star_{i_1 ,0} \ge \beta_0$, due to the fact that $\overline{\phi^\star} =\beta_0$. In turn, a careful evaluation of~\eqref{CHDBC-positive-2}, with the direction $\psi$ given by~\eqref{CHDBC-positive-2-2}, gives 
	\begin{align}
\frac{1}{h^2}\diff_\tau \! \left.F^n_h(\phi^\star +\tau\psi ) \right|_{\tau=0} & = \frac{1}{s } \Big( L_h^{-1} ( \phi^\star-\phi^n )_{i_0,j_0} - \frac12 L_h^{-1} ( \phi^\star-\phi^n )_{i_1,0} \Big)  
   	\nonumber 
	\\
& \quad   - \frac{2}{hs} (- \Delta_h^x)^{-1} ( \phi_B^\star - \phi_B^n )_{i_1}  
 + \varepsilon^{-1}  \Big(  \ln ( 1+ \phi^\star_{i_0,j_0}) 
   - \ln ( 1- \phi^\star_{i_0,j_0} )  \Big) 
	\nonumber 
	\\
& \quad  + \varepsilon^{-1}  ( \frac12 + 2h^{-1} ) \Big( - \ln ( 1+ \phi^\star_{i_1, 0} ) + \ln ( 1- \phi^\star_{i_1, 0} )   \Big)   
	\nonumber 
	\\ 	
& \quad + \varepsilon ( L_h \phi^\star_{i_0,j_0}  - \frac12 L_h \phi^\star_{i_1, 0}  ) 
 + \kappa h^{-1}  \Delta_h^x \phi^\star_{i_1}     
	\nonumber 
	\\
& \quad 
  - \varepsilon^{-1} \theta_0  \phi^n_{i_0,j_0} 
  + \varepsilon^{-1} \theta_0 ( \frac12 + 2h^{-1} )  \phi^n_{i_1,0}   .
	\label{CHDBC-positive-15} 
	\end{align} 
Using similar estimates as in the first two cases, we are able to prove that 
	\begin{equation} 
\frac{1}{h^2}\diff_\tau \! \left.F_h^n(\phi^\star +\tau\psi )\right|_{\tau=0} \le \varepsilon^{-1} \ln\left( \frac{\delta}{2+\delta}\right)  + C_5 < 0 , 
	\label{CHDBC-positive-15-2} 
	\end{equation}
in which $C_5$ stands for another constant for fixed $s >0$, $h>0$ and $\varepsilon >0$. The technical details are skipped for the sake of brevity. Of course, such an inequality contradicts the assumption that $F^n_h$ has a minimum at $\phi^\star$, since the directional derivative is negative in a direction pointing into the interior of $A_{h,\delta}$. 

As a result, global minimum of $F^n_h$ over $A_{h,\delta}$ could not occur at a boundary point $\phi^\star$ such that  $\phi^\star_{i_0,j_0}  = 1-\delta$, for some $(i_0, j_0)$, so that the grid function $\phi^\star$ has a global maximum at $(i_0, j_0)$. The analysis follows similar ideas as outlined above, and the details are left to interested readers. 

Therefore, the global minimum of $F^n_h$ over $A_{h,\delta}$ could only possibly occur at an interior point, for $\delta>0$ sufficiently small. Since $F^n_h$ is a smooth function, we conclude that there must be a solution $\phi \in A_{h, \delta}$ (provided that $\delta$ is sufficiently small), so that
	\begin{equation} 
\diff_\tau \! \left.F^n_h(q +\tau\psi )\right|_{\tau=0} =0 ,  \quad \forall \, \psi \in H,
	\label{CHDBC-positive-16} 
	\end{equation} 
which is equivalent to the numerical solution of~\eqref{scheme-CHDBC-rewrite-1}.  Therefore, there exists a numerical solution to~\eqref{scheme-CHDBC-rewrite-1}, over the compact domain $A_{h,\delta} \subset A_{h}$, with point-wise positive values for $1 + \phi^{n+1}$ and $1 - \phi^{n+1}$. The existence of a positive numerical solution is established. 

Meanwhile, since $F_h^n$ is a strictly convex function over $A_h$, the uniqueness analysis for this numerical solution (over the open set $A_h$) is straightforward, following a convexity analysis. The proof of Theorem~\ref{CHDBC-positivity} is complete.  
	\end{proof} 
	
	\begin{rem} 
 The positivity preserving analysis used herein has been successfully applied to various gradient flow models, such as the Cahn-Hilliard equation with Flory-Huggins energy potential~\cite{chen22b, chen22a, chen19b, dong19b, Dong2021a, Dong2022a, dong20a, Yuan2021a, Yuan2022a}, the liquid film droplet model~\cite{ZhangJ2021}, the Poisson-Nernst-Planck system~\cite{LiuC2021a, LiuC2022a, QianWangZhou_JCP20}, the reaction-diffusion system with detailed balance~\cite{LiuC2021b, LiuC2022c, LiuC2022b}, etc. In these works, the convex nature of the energy functional associated with singular term has played an essential role. This feature prevents the numerical solution approach the singular limit values of $-1$ and 1,  which turns out to be the key point in the analysis. 
\end{rem} 
 
	\section{Total energy stability analysis}
	\label{sec:energy stability}

With the positivity-preserving and unique solvability properties for the numerical scheme~\eqref{scheme-CHDBC-1} -- \eqref{scheme-CHDBC-8} established, a stability analysis for the total energy could be established. The following discrete energy is introduced: 
	\begin{align} 
E_h (\phi) & := \varepsilon^{-1} \Big( \cipgen{I(\phi)}{1}{\Omega}+ \cipgen{I(\phi_B)}{1}{\Gamma}  + \cipgen{I(\phi_T}{1}{\Gamma} \Big)  - \frac{\theta_0}{2\varepsilon} \Big( \| \phi \|_{2,\Omega}^2 + \| \phi_B \|_{2,\Gamma}^2 + \| \phi_T \|_{2,\Gamma}^2  \Big) 
  	\nonumber
  	\\
& \quad  +  \frac{\varepsilon}{2} \cipgen{\phi}{L_h\phi}{\Omega} - \frac{\kappa}{2} \Big( \cipgen{\phi_B}{\Delta_h^x\phi_B}{\Gamma} 
  + \cipgen{\phi_T}{\Delta_h^x\phi_T}{\Gamma}   \Big) ,
	\label{CHDBC-discrete energy}
	\end{align} 
where 
	\[
I(\phi) = (1+\phi)\ln(1+\phi) + (1-\phi)\ln(1-\phi).	
	\]

	\begin{thm}
	\label{CHDBC-energy stability} 
For any time step size $s>0$, the numerical solution of \eqref{scheme-CHDBC-1} -- \eqref{scheme-CHDBC-8} satisfies 
	\begin{equation} 
  E_h (\phi^{n+1} ) 
+ s  ( \| \nabla_h \mu^{n+1} \|_{2,\Omega}^2 +  \| D_x \mu_B^{n+1} \|_{2,\Gamma}^2 +\| D_x \mu_T^{n+1} \|_{2,\Gamma}^2 ) \le E_h (\phi^n)   ,   
	\label{CHDBC-energy-0} 
	\end{equation} 
so that $E_h (\phi^n) \le E_h (\phi^0)$, for all $n \in \mathbb{N}$. 
	\end{thm}
	\begin{proof} 
Taking a discrete inner product of~\eqref{scheme-CHDBC-1} with $\mu^{n+1}$ gives   
	\begin{align} 
 \langle \phi^{n+1} - \phi^n , \mu^{n+1} \rangle_\Omega &= \cipgen{  \phi^{n+1} - \phi^n}{ \varepsilon^{-1} \left( I'(\phi^{n+1}) - \theta_0 \phi^n\right) - \varepsilon \Delta_h \phi^{n+1}}{\Omega} 
 	\nonumber
 	\\
 & = s\cipgen{\mu^{n+1}}{\Delta_h\mu^{n+1}}{\Omega}
 	\nonumber 
	\\
& = - s  \| \nabla_h \mu^{n+1} \|_{2,\Omega}^2  ,  
	\label{CHDBC-energy-1}      
	\end{align} 
in which the homogeneous discrete Neumann boundary conditions for $\mu^{n+1}$ have been applied. Next, the convexity of $I(\phi)$, implies that 
	\begin{equation} 
\cipgen{\phi^{n+1} - \phi^n}{I'(\phi^{n+1})}{\Omega}  \ge \cipgen{I(\phi^{n+1})}{1}{\Omega} -  \cipgen{I(\phi^n)}{1}{\Omega}
	\label{CHDBC-energy-2-2}  
	\end{equation}	
	\begin{equation} 
\cipgen{\phi_B^{n+1} - \phi_B^n}{I'(\phi_B^{n+1})}{\Gamma}  \ge \cipgen{I(\phi_B^{n+1})}{1}{\Gamma} -  \cipgen{I(\phi_B^n)}{1}{\Gamma}
	\label{CHDBC-energy-2-2-B}  
	\end{equation}
	\begin{equation} 
\cipgen{\phi_T^{n+1} - \phi_T^n}{I'(\phi_T^{n+1})}{\Gamma}  \ge \cipgen{I(\phi_T^{n+1})}{1}{\Gamma} -  \cipgen{I(\phi_T^n)}{1}{\Gamma}
	\label{CHDBC-energy-2-2-T}  
	\end{equation}
	\begin{equation}
\cipgen{\phi^{n+1} - \phi^n}{- \phi^n}{\Omega} \ge  - \frac12 ( \| \phi^{n+1} \|_{2,\Omega}^2 - \| \phi^n \|_{2,\Omega}^2 ) .
	\label{CHDBC-energy-2-3}
	\end{equation}
	\begin{equation}
\cipgen{\phi_B^{n+1} - \phi_B^n}{- \phi_B^n}{\Gamma} \ge  - \frac12 ( \| \phi_B^{n+1} \|_{2,\Gamma}^2 - \| \phi_B^n \|_{2,\Gamma}^2 ) .
	\label{CHDBC-energy-2-3-B}
	\end{equation}
	\begin{equation}
\cipgen{\phi_T^{n+1} - \phi_T^n}{- \phi_T^n}{\Gamma} \ge  - \frac12 ( \| \phi_T^{n+1} \|_{2,\Gamma}^2 - \| \phi_T^n \|_{2,\Gamma}^2 ) .
	\label{CHDBC-energy-2-3-T}
	\end{equation}
Meanwhile, for the surface diffusion term, an application of the summation-by-parts identity~\eqref{lemma 1-0-3} (in Lemma~\ref{lemma1}) indicates that 
	\begin{align} 
- \cipgen{\phi^{n+1} - \phi^n}{\Delta_h \phi^{n+1} }{\Omega} & =  \eipvec{\nabla_h ( \phi^{n+1} - \phi^n ) }{ \nabla_h\phi^{n+1}}_{\Omega}  
	\nonumber 
	\\
& \quad - \cipgen{ \tilde{D}_{y} \phi^{n+1}_{\star, N}}{ \phi^{n+1}_{\star,N} - \phi^n_{\star,N}}{\Gamma} + \cipgen{\tilde{D}_{y} \phi^{n+1}_{\star, 0}}{ \phi^{n+1}_{\star,0} - \phi^n_{\star,0}}{\Gamma} .      
	\label{CHDBC-energy-3-1} 
	\end{align} 
The estimate for the first term on the right hand side of~\eqref{CHDBC-energy-3-1} comes from the convexity of $\| \nabla_h \phi \|_{2,\Omega}^2$: 
	\begin{equation}
\eipvec{\nabla_h ( \phi^{n+1} - \phi^n ) }{ \nabla_h\phi^{n+1}}_\Gamma \ge  \frac12 \left( \| \nabla_h \phi^{n+1} \|_{2,\Omega}^2 - \| \nabla_h \phi^n \|_{2,\Omega}^2 \right) . 
	\label{CHDBC-energy-3-2}
	\end{equation}
For the boundary terms appearing in~\eqref{CHDBC-energy-3-1}, we recall the representation formula~\eqref{scheme-CHDBC-rewrite-2} and obtain 	
	\begin{align}  
\varepsilon \cipgen{\tilde{D}_{y} \phi^{n+1}_{\star, 0}}{ \phi^{n+1}_{\star,0} - \phi^n_{\star,0}}{\Gamma} &	= \varepsilon \cipgen{\tilde{D}_{y} \phi^{n+1}_{\star, 0}}{ \phi^{n+1}_B - \phi^n_B}{\Gamma}   
	\nonumber 
	\\        
& = \frac{1}{s } \cipgen{( - \Delta_h^x )^{-1} ( \phi^{n+1}_B - \phi^n_B )}{\phi^{n+1}_B - \phi^n_B}{\Gamma}    
	\nonumber 
	\\
& \quad + \cipgen{ \varepsilon^{-1} ( I'(\phi^{n+1}_T) - \theta_0 \phi^n_T ) - \kappa \Delta_h^x \phi^{n+1}_T}{ \phi^{n+1}_T - \phi^n_T}{\Gamma} . 	
	 \label{CHDBC-energy-3-3-B}	
	\end{align}
Similarly,
	\begin{align}  
-\varepsilon \cipgen{\tilde{D}_{y} \phi^{n+1}_{\star, N}}{ \phi^{n+1}_{\star,N} - \phi^n_{\star,N}}{\Gamma} & = - \varepsilon \cipgen{\tilde{D}_{y} \phi^{n+1}_{\star,N}}{ \phi^{n+1}_T - \phi^n_T}{\Gamma}   
	\nonumber 
	\\        
& = \frac{1}{s} \cipgen{( - \Delta_h^x )^{-1} ( \phi^{n+1}_T - \phi^n_T )}{\phi^{n+1}_T - \phi^n_T}{\Gamma}    
	\nonumber 
	\\
& \quad + \cipgen{ \varepsilon^{-1} ( I'(\phi^{n+1}_T) - \theta_0 \phi^n_T)  - \kappa \Delta_h^x \phi^{n+1}_T}{ \phi^{n+1}_T - \phi^n_T}{\Gamma} . 	
	 \label{CHDBC-energy-3-3-T}	
	\end{align}

We also have
	\begin{align} 
\frac{1}{s } \cipgen{ ( - \Delta_h^x )^{-1} ( \phi^{n+1}_B - \phi^n_B ) }{  \phi^{n+1}_B - \phi^n_B}{\Gamma} & =  - s  \cipgen{ \mu_B^{n+1}}{\Delta_h^x \mu_B^{n+1}}{\Gamma}
	\nonumber 
	\\
& = s  \| D_x \mu_B^{n+1} \|_{2, \Gamma}^2 , 
	\label{CHDBC-energy-3-4-B}
	\end{align}
	\begin{equation} 
\cipgen{\phi^{n+1}_B - \phi^n_B}{ - \Delta_h^x \phi^{n+1}_B }{\Gamma}  \ge \frac12 ( \| D_x \phi^{n+1}_B \|_{2, \Gamma}^2 
	 - \| D_x \phi^n_B  \|_{2 ,\Gamma}^2 ) , 
	 	\label{CHDBC-energy-3-8-B}	 
	 \end{equation}  
and similarly,
	\begin{align} 
\frac{1}{s} \cipgen{ ( - \Delta_h^x )^{-1} ( \phi^{n+1}_T - \phi^n_T ) }{  \phi^{n+1}_T - \phi^n_T}{\Gamma} & =  - s  \cipgen{ \mu_T^{n+1}}{\Delta_h^x \mu_B^{n+1}}{\Gamma}
	\nonumber 
	\\
& = s  \| D_x \mu_T^{n+1} \|_{2, \Gamma}^2 , 
	\label{CHDBC-energy-3-4-T}
	\end{align}
	\begin{equation} 
\cipgen{\phi^{n+1}_T - \phi^n_T}{ - \Delta_h^x \phi^{n+1}_T }{\Gamma}  \ge \frac12 ( \| D_x \phi^{n+1}_T \|_{2, \Gamma}^2 
	 - \| D_x \phi^n_T  \|_{2 ,\Gamma}^2 ) , 
	 	\label{CHDBC-energy-3-8-T}	 
	 \end{equation}  
Then we arrive at, for the bottom boundary,
	\begin{align}  
& \hspace{-0.25in} \varepsilon\cipgen{\tilde{D}_{y} \phi^{n+1}_{\star, 0}}{\phi^{n+1}_{\star,0} - \phi^n_{\star,0}}{\Gamma}
	\nonumber
	\\
& \ge  \varepsilon^{-1} I(\phi_B^{n+1}) - \varepsilon^{-1}I(\phi^n_B) - \frac{\theta_0}{2} ( \| \phi^{n+1}_B \|_{2, \Gamma}^2 
	 - \| \phi^n_B  \|_{2 ,\Gamma}^2 ) 
	\nonumber
	\\    
& \quad + \frac{\varepsilon}{2} \left( \| D_x \phi^{n+1}_B \|_{2, \Gamma}^2 - \| D_x \phi^n_B  \|_{2 ,\Gamma}^2 \right) + s  \| D_x \mu_B^{n+1} \|_{2, \Gamma}^2 . 	
	 \label{CHDBC-energy-3-9}	
	\end{align}
and, for the top boundary,
	\begin{align}  
& \hspace{-0.25in} - \varepsilon\cipgen{\tilde{D}_{y} \phi^{n+1}_{\star, N}}{\phi^{n+1}_{\star,N} - \phi^n_{\star,N}}{\Gamma}
	\nonumber
	\\
& \ge  \varepsilon^{-1} I(\phi_T^{n+1}) - \varepsilon^{-1}I(\phi^n_T) - \frac{\theta_0}{2} ( \| \phi^{n+1}_T \|_{2, \Gamma}^2 - \| \phi^n_T  \|_{2 ,\Gamma}^2 ) 
	\nonumber
	\\    
& \quad + \frac{\varepsilon}{2} \left( \| D_x \phi^{n+1}_T \|_{2, \Gamma}^2 - \| D_x \phi^n_T  \|_{2 ,\Gamma}^2 \right) + s  \| D_x \mu_T^{n+1} \|_{2, \Gamma}^2 . 	
	 \label{CHDBC-energy-3-10}	
	\end{align}

Finally, substitution of~\eqref{CHDBC-energy-2-2}, \eqref{CHDBC-energy-2-3}, \eqref{CHDBC-energy-3-1}, \eqref{CHDBC-energy-3-2}, \eqref{CHDBC-energy-3-9} and \eqref{CHDBC-energy-3-10} into~\eqref{CHDBC-energy-1} results in~\eqref{CHDBC-energy-0}, so that the total energy stability is proved. This finishes the proof of Theorem~\ref{CHDBC-energy stability}. 
	\end{proof} 
\begin{rem} 
	Many studies have discussed the convergence of convex splitting schemes under classical periodical and homogeneous Neumann boundary conditions. Since the spatial truncation error, the exact solution is inconsistent with the discrete mass conservation. A Fourier projection (including sine or cosine projection) is commonly applied to treat the issue with acceptable artificial error \cite{Chen16,chen19b,Guo24}. For the semi-discrete convergence proposed in \cite{Bao2021b,Meng23}, this problem does not happen. 
    However, the Fourier projection is invalid for the dynamical boundary condition. A suitable alternative to the Fourier projection is needed, which is the difficulty in analyzing the convergence of the proposed scheme. We will consider this issue in our future research.
\end{rem} 

\section{Numerical experiments}  \label{sec:numerical results} 

In this section, some numerical examples will be presented to validate the proposed scheme. A combination of multigrid method and Newton iteration is applied to deal with the nonlinear equations coupled between the interior region and the boundary section. 
A nonlinear full approximation scheme (FAS) multigrid method is applied for solving the semi-implicit numerical scheme~\eqref{scheme-CHDBC-1} -- \eqref{scheme-CHDBC-5}. We present the convergence test and perform some 2-D sample computation results, while a 3-D extension could be similar. The computational domain is taken as $\Omega=(0, 1)^2$, parameters will be specified when mentioned.

\subsection{Convergence test}
We first test the convergence rates of the scheme~\eqref{scheme-CHDBC-1} -- \eqref{scheme-CHDBC-5}. For this numerical experience, we take~$s =0.001h^2$~so that the target convergence rate is of $O(h^2)$. Parameters are fixed as~$\varepsilon =0.02$, $\kappa=0.02$~and~$\theta_0=3$. All the programs stop at the same final time. A smooth initial data is chosen as
\begin{equation}
\phi_0 (x,y) =0.8\cos ( 4\pi x ) \cos ( 4\pi y ) .
\end{equation}
The convergence rate for the physical boundary section is indicated separately to verify the 2nd order spatial accuracy in both the interior area and boundary section. Associated norms in different discrete function spaces and different dimensions, defined in subsection~\ref{subsec:finite difference}, are displayed respectively. Since the exact form of the solution is not available, instead of calculating the absolute error at final time, we compute the Cauchy difference, instead of a direct calculation of the numerical error: 
\begin{equation}
\delta_{\phi}:=\phi_{h_c}-\mathcal{I}_f^c\left(\phi_{h_f}\right), \label{Cauthy difference}
\end{equation}
where~$\mathcal{I}_f^c$~is a projection from coarse mesh to fine mesh. This requires relatively coarse solution, parametrized by $h_c$, and a relatively fine solution, parametrized by~$h_f$, where~$h_c=2h_f$, at the same final time. The detailed results are displayed in the following tables.
\begin{table}[h]
\center
\begin{tabular}{lllll}
\multicolumn{5}{c}{Boundary}                                      \\ \hline
Grid size           & 16-32       & 32-64      & 64-128     & 128-256    \\ \hline
$\ell^2$            & 3.0401E-02  & 7.6736E-03 & 1.9241E-03 & 4.8129E-04 \\
$\ell^2$  rate      &             & 1.9861     & 1.9957     & 1.9992     \\
$\ell^\infty$       & 2.7768E-02  & 6.9293E-03 & 1.7708E-03 & 4.4370E-04 \\
$\ell^\infty$ rate  &             & 2.0027     & 1.9683     & 1.9967     \\ \hline
\end{tabular}
\caption{Convergence rate on the physical boundary}\label{boundary rate}
\end{table}
\begin{table}[h]
\center
\begin{tabular}{lllll}
\multicolumn{5}{c}{The whole domain}                                      \\ \hline
Grid size           & 16-32       & 32-64      & 64-128     & 128-256    \\ \hline
$\ell^2$            & 1.7003E-02  & 4.0737E-03 & 9.9197E-04 & 2.4604E-04 \\
$\ell^2$  rate      &             & 2.1024     & 2.0569     & 2.0113     \\
$\ell^\infty$       & 2.7768E-02  & 7.7649E-03 & 2.1448E-03 & 6.9708E-04 \\
$\ell^\infty$ rate  &             & 1.8384     & 1.8561     & 1.6214     \\ \hline
\end{tabular}
\caption{Convergence rate over the whole domain}\label{total rate}
\end{table}

Table~\ref{boundary rate} and Table~\ref{total rate} present the errors and convergence rates, based on the data at~$T=0.001$~produced by the numerical scheme~\eqref{scheme-CHDBC-1} -- \eqref{scheme-CHDBC-5}. The 2nd order convergence on the physical boundary is indicated in Table~\ref{boundary rate}, where the inner product and the associated norm are defined in the lower dimension. It is observed that all the numerical convergence rates approach to the theoretical value as the spatial mesh is refined.

\subsection{Spinodal decomposition}

In this section, we perform a numerical simulation to verify energy decay and mass conservation.  The parameters are chosen as follows: $\varepsilon=0.02, \theta_0=3, \kappa=1$.  The time step size and space step are taken as $h=1/128, s=10^{-5}$, respectively. 

To simulate the phenomenon of spinodal decomposition, the initial data in the bulk are chosen to satisfy
\begin{equation}
\phi^0_{i,j}=0.2+0.02 \times r_{i,j},\label{rand_initial}
\end{equation}
where $r_{i,j}$ are uniformly distributed random numbers in $[-1,1]$. In Figure \ref{rand_figure} we present several snapshot evolutions of $\phi$ at the selected time instants. The physical boundary is established on the left and right sides of the square domain, while a periodic boundary condition is imposed in the $y$ direction. 

The mixture does not phase separate on the physical boundary because there are no initial fluctuations there, and $\phi$ takes on the uniform value of $0.2$ at the left and right boundaries. In the snapshots, the interface is not perpendicular to the wall, in contrast  to the case where homogeneous Neumann boundary conditions are taken on the physical boundary. A transition layer at the physical boundary forms, with a certain thickness, where a short-range wetting-type interaction can be observed. 

In the left part of Figure \ref{rand_energy}, the evolution of energy is illustrated, which indicates an energy decay in time. We also display the evolution of mass difference of $\phi$, which is defined as $\bar{\phi}^n-\bar{\phi}^0$, with $\bar{\phi}^n$ given by~\eqref{inner_mass}. The mass variation is displayed in the right part of Figure \ref{rand_energy}, and indicates mass conservation of about $1\times 10^{-9}$. 

\newpage
\begin{figure}[ht]
\center\hspace{-5mm}
\subfigure[t=0.0005]{
\includegraphics[scale=0.4]{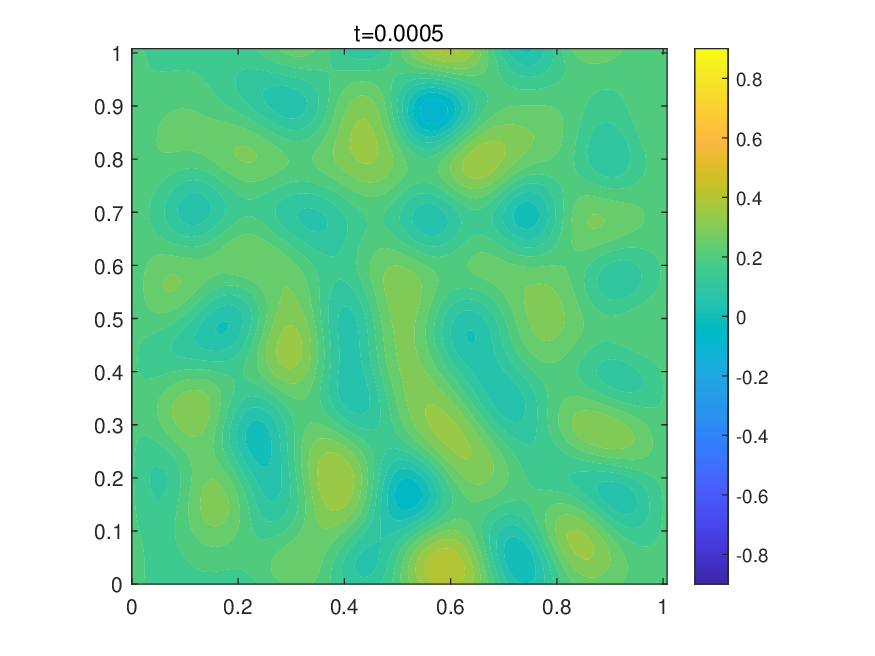}
}\hspace{-10mm}
\subfigure[t=0.0008]{
\includegraphics[scale=0.4]{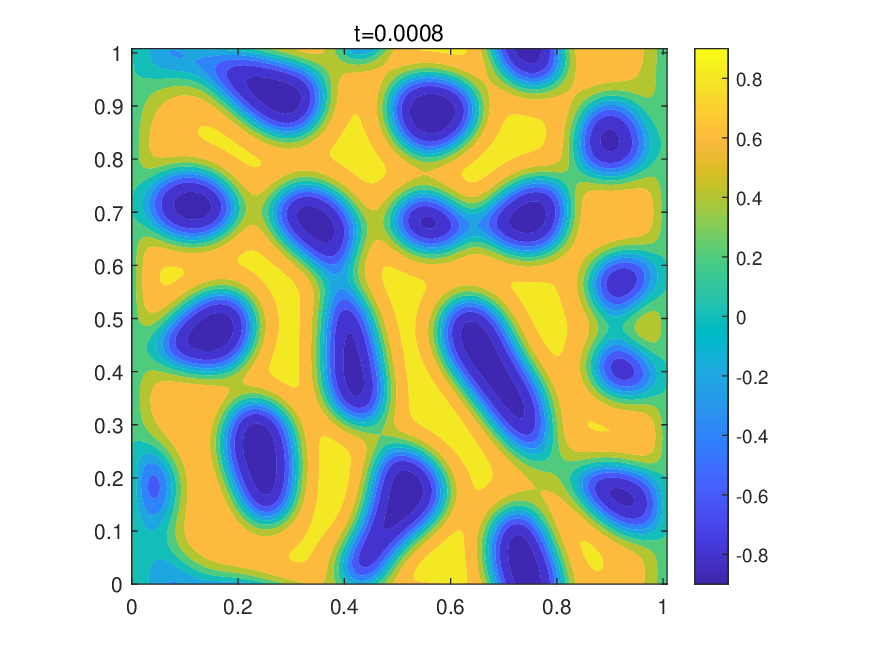}
}\hspace{-10mm}
\subfigure[t=0.001]{
\includegraphics[scale=0.4]{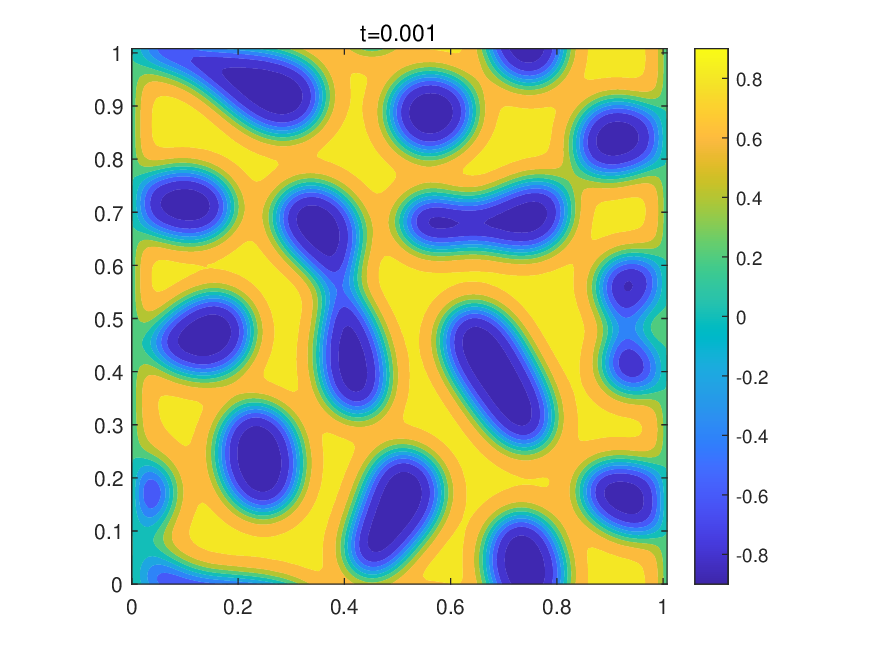}
}

\hspace{-5mm}\subfigure[t=0.002]{
\includegraphics[scale=0.4]{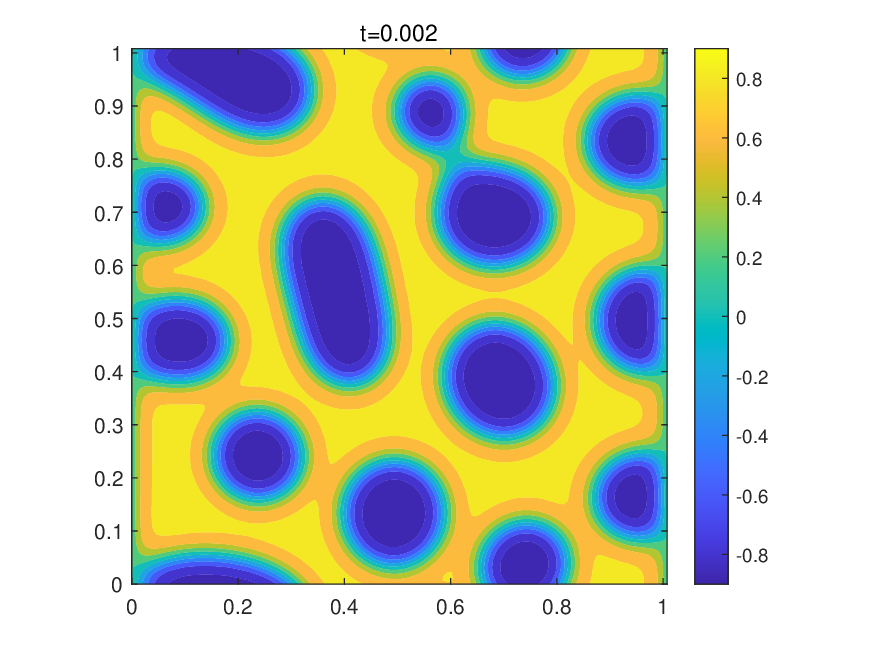}
}\hspace{-10mm}
\subfigure[t=0.005]{
\includegraphics[scale=0.4]{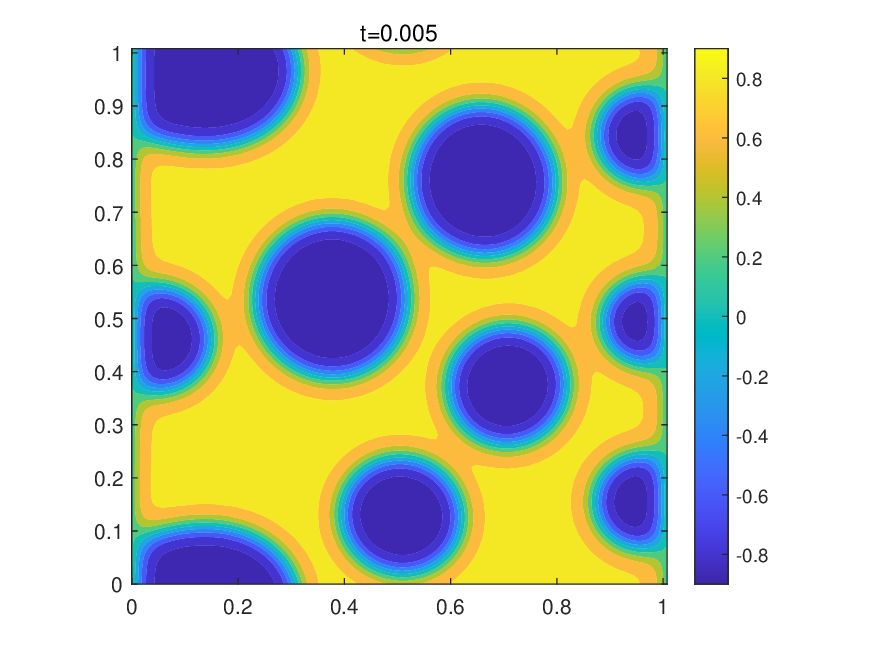}
}\hspace{-10mm}
\subfigure[t=0.01]{
\includegraphics[scale=0.4]{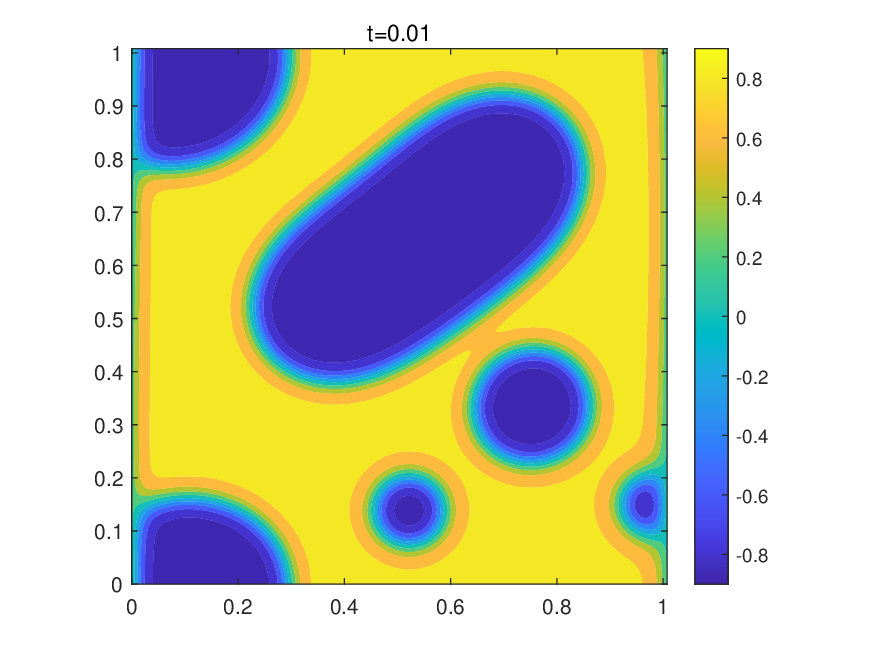}
}

\hspace{-5mm}\subfigure[t=0.02]{
\includegraphics[scale=0.4]{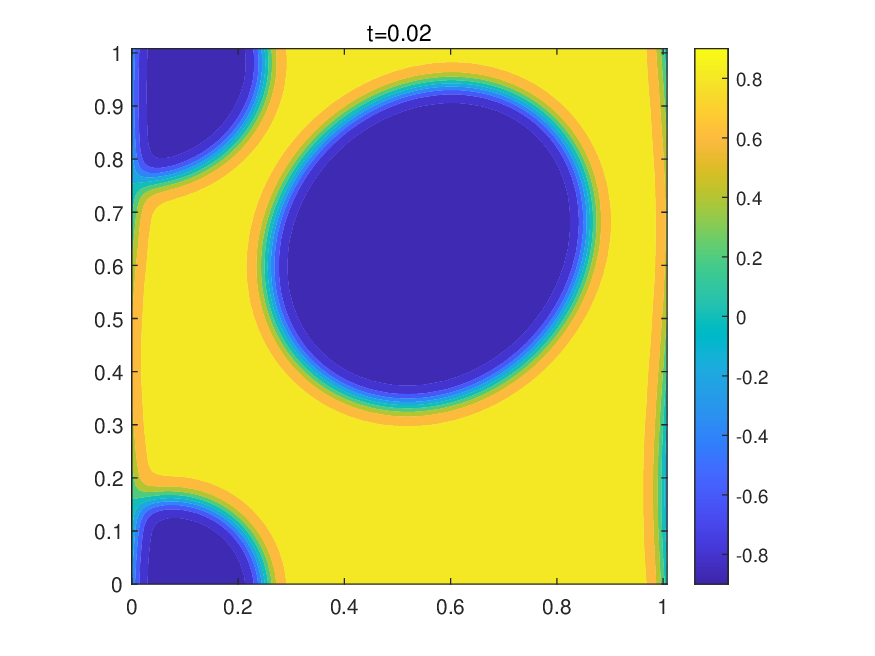}
}\hspace{-10mm}
\subfigure[t=0.05]{
\includegraphics[scale=0.4]{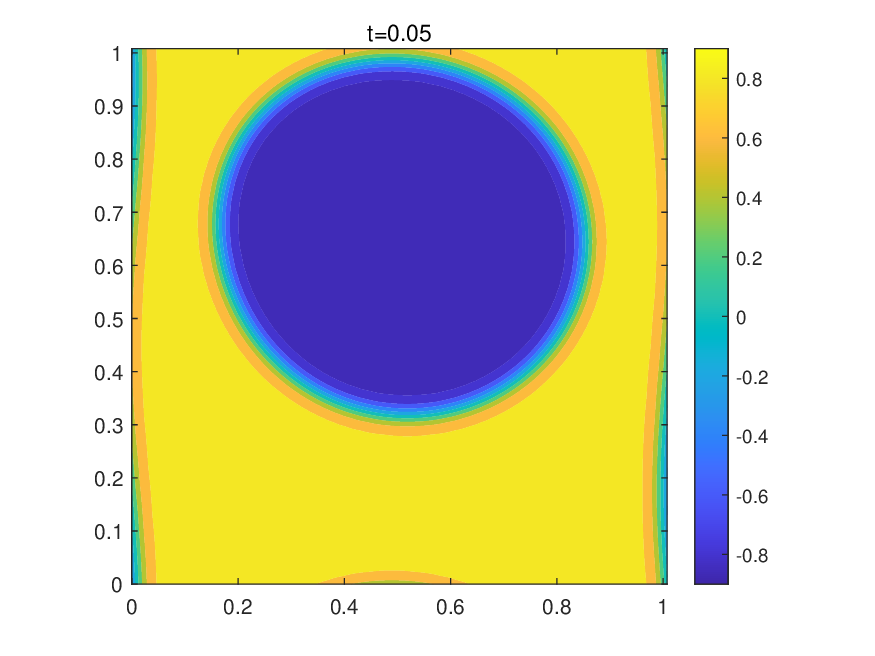}
}\hspace{-10mm}
\subfigure[t=0.1]{
\includegraphics[scale=0.4]{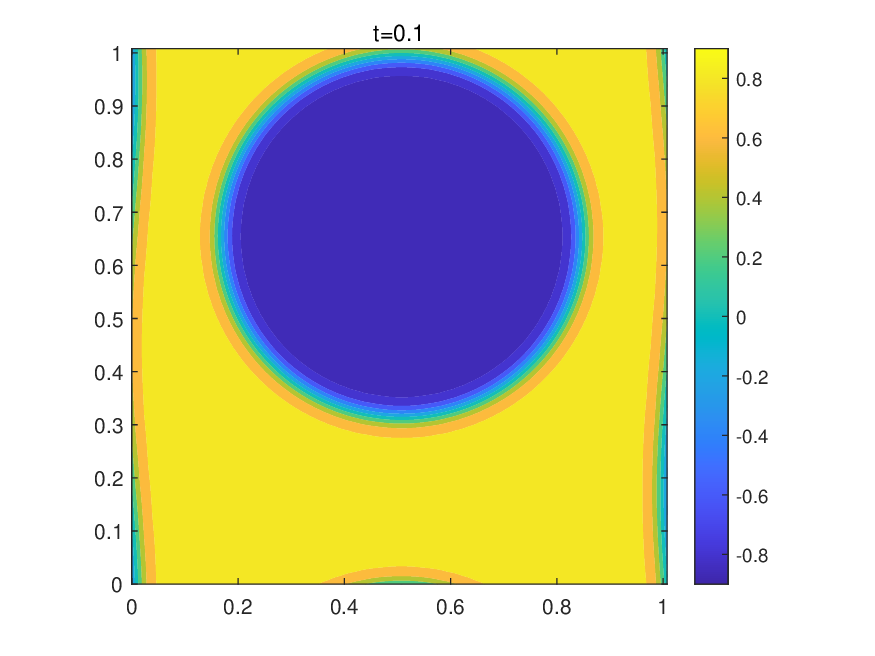}
}
\caption{The phase evolution of $\phi$ at several time instants with initial data~\eqref{rand_initial}.}\label{rand_figure}
\end{figure}
\begin{figure}[!h]
\center
\subfigure[Energy]{
\includegraphics[height=4.5cm,width=6cm]{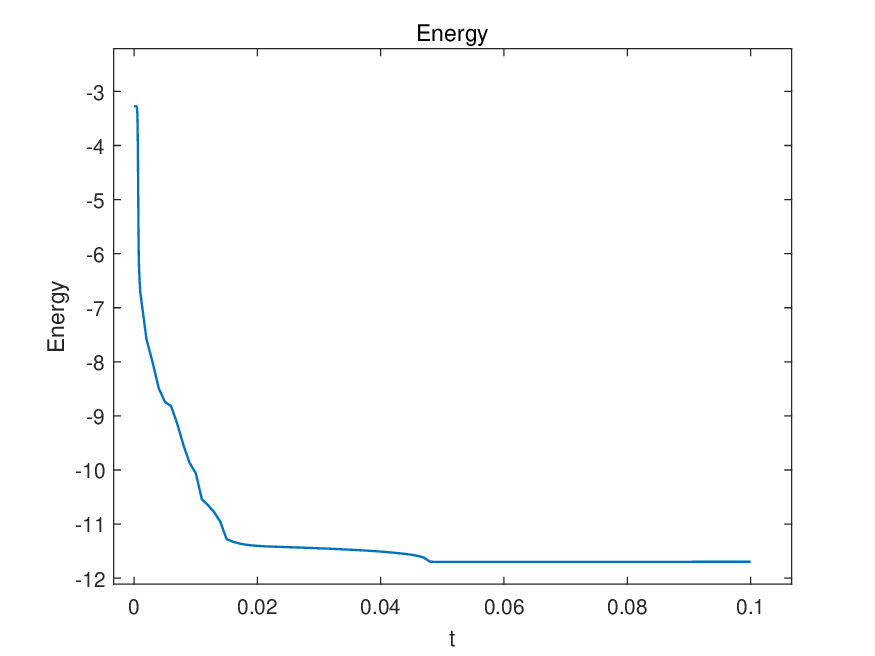}
}
\subfigure[Mass difference]{
\includegraphics[height=4.5cm,width=6cm]{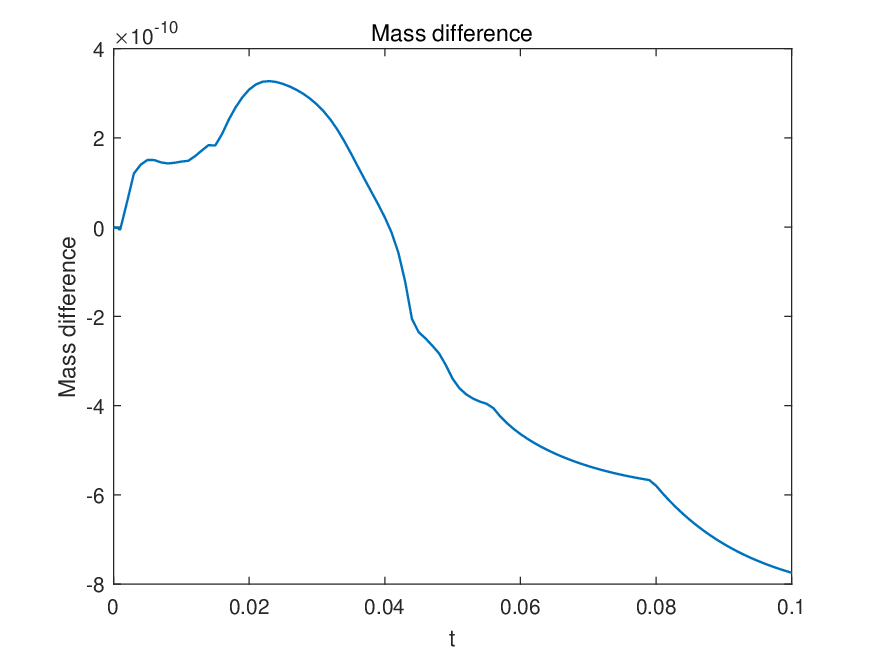}
}
\caption{(a) Left: the time evolution of the energy. (b) Right: the time evolution of the mass difference.} \label{rand_energy}
\end{figure}

\subsection{Droplet evolution}
In this part, we initialize a square droplet with its center at point $(0.5, 0.2)$, and one edge of the droplet occurs at the dynamical boundary (see Figure \ref{fig: droplet}(a)). The phase value inside the droplet is $0.8$ and outside the droplet is $-0.8$, and the parameters are chosen as $\varepsilon=0.01, \theta_0=3, \kappa=0.01$. Figure \ref{fig: droplet} presents the evolution of the droplet. The physical, mass-conserving boundary is now at the bottom of the domain. It is observed that the square droplet gradually evolves into a circle and the interface forms a certain contact angle with the solid wall that deviates from 90 degrees. Recall that a 90-degree angle would form at a physical boundary when homogeneous Neumann boundary conditions are used. The deviation that is observed is due to the fact that mass is separately conserved on the physical boundary, and this prevents the formation of perpendicular contact angle.

\begin{figure}[!h]
\center\hspace{-5mm}
\subfigure[t=0]{
\includegraphics[scale=0.4]{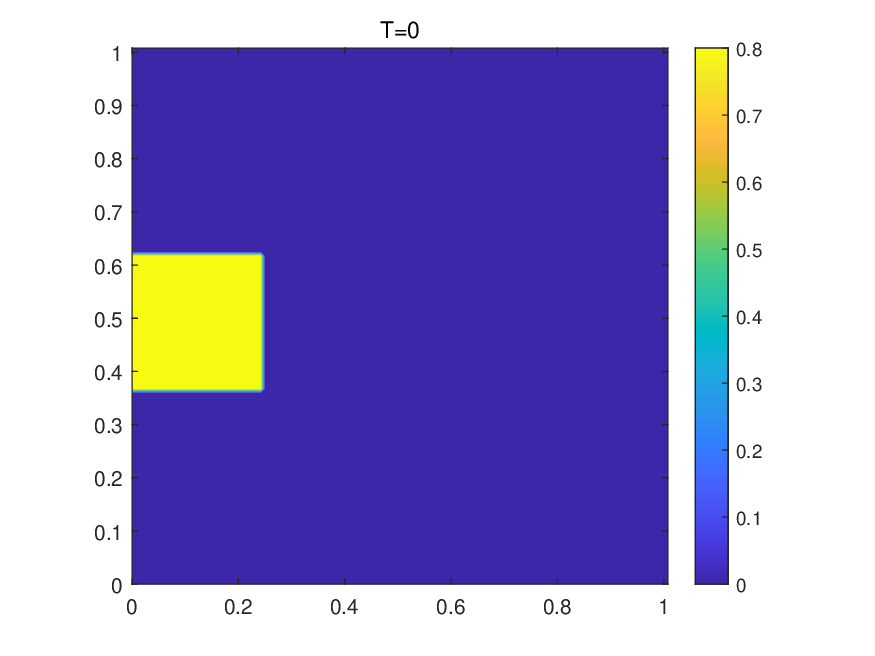}
}\hspace{-10mm}
\subfigure[t=0.001]{
\includegraphics[scale=0.4]{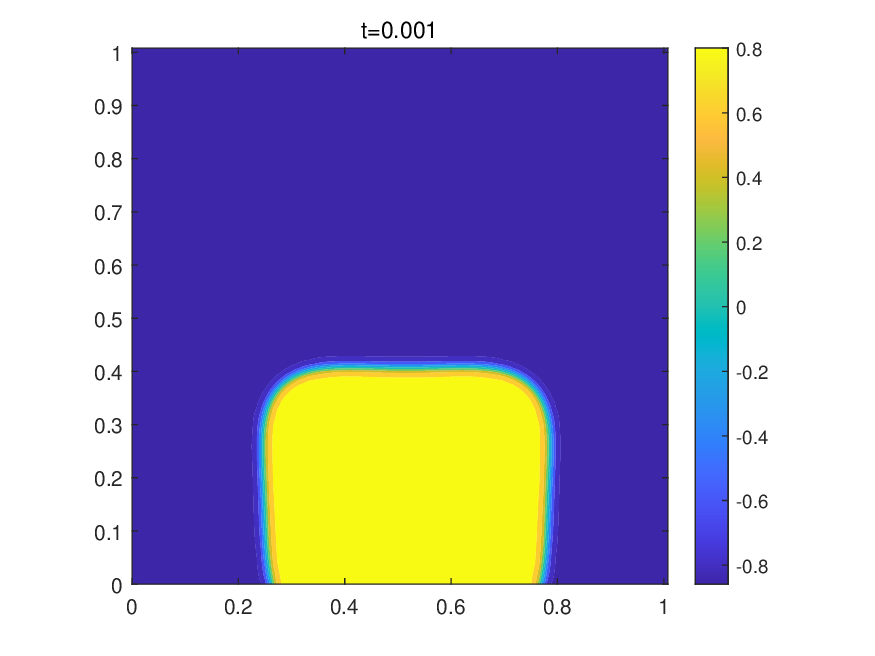}
}\hspace{-10mm}
\subfigure[t=0.003]{
\includegraphics[scale=0.4]{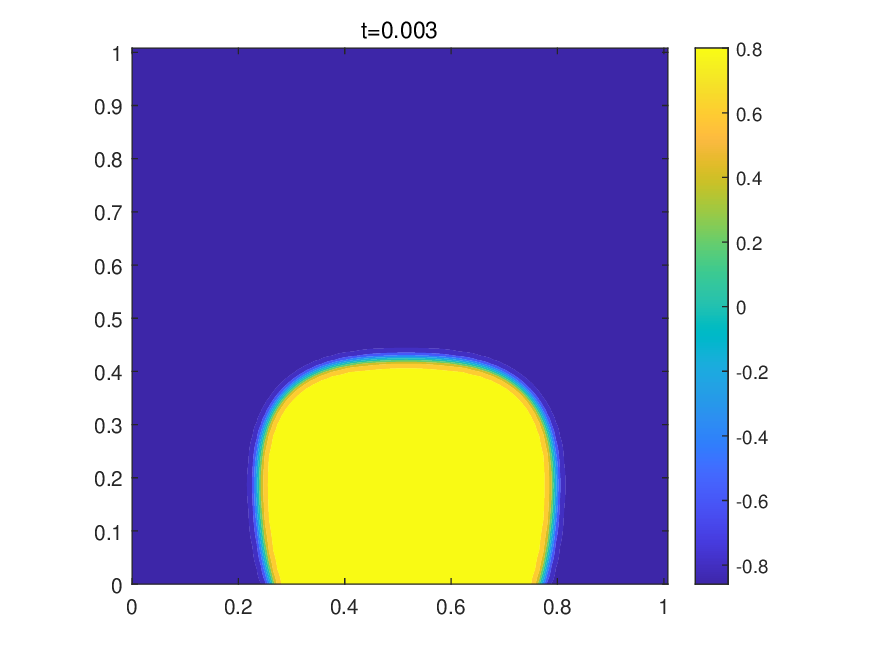}
}

\subfigure[t=0.005]{
\includegraphics[scale=0.4]{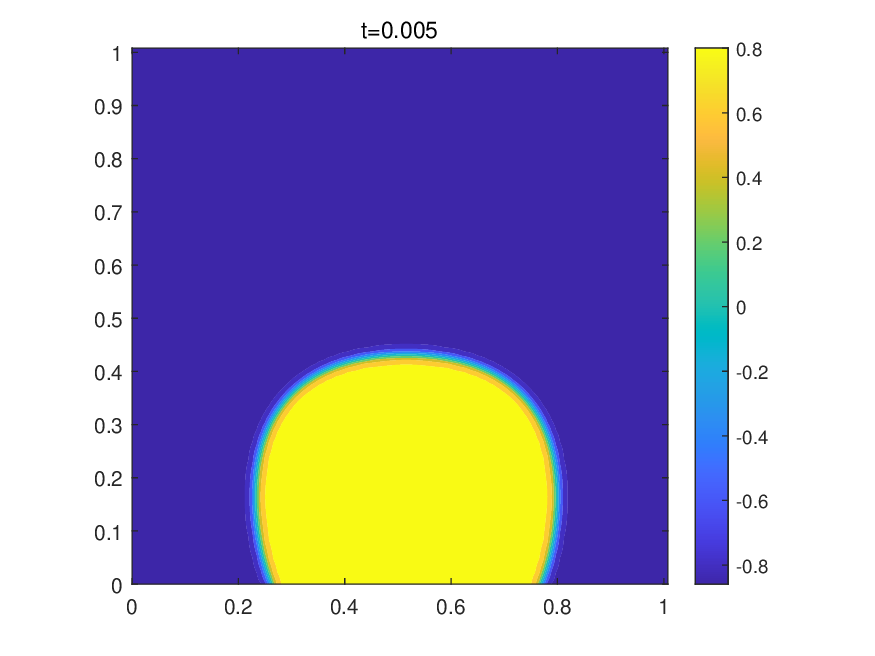}
}
\subfigure[t=0.01]{
\includegraphics[scale=0.4]{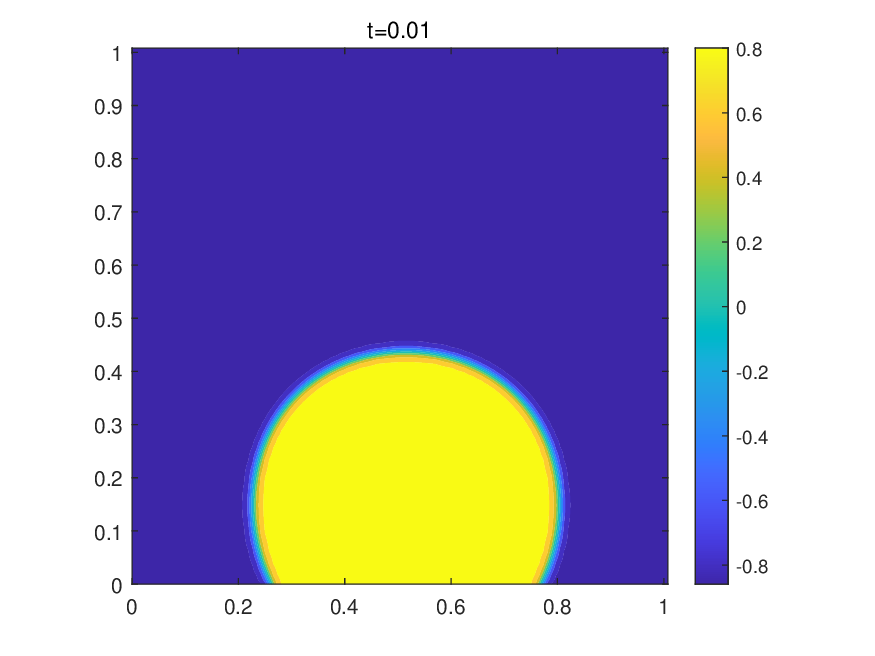}
}
\caption{Snapshots of solution for droplet evolution at selected times.} 
\label{fig: droplet}
\end{figure}

\subsection{Two droplet evolution} \label{Two droplets evolution}
In this part, we numerically simulate the slow fusion process of two circular droplets. As shown in Figure \ref{fig:ini two droplets}, two droplets barely contact with each other and the bottom boundary crosses through both of them. Two numerical simulations, with homogeneous Neumann and dynamical boundary conditions, respectively, are displayed in Figure \ref{fig: two droplets}. In the case of homogeneous Neumann boundary condition, the interface between two kinds of liquid (yellow and blue) quickly becomes perpendicular to the boundary. Without mass conservation on the surface, the yellow part on the boundary merges into one very fast and lengthened slightly. In the case of dynamical boundary condition, the total length of the yellow part basically remains unchanged, due to the mass conservation on the surface, and the coarsening on boundary takes a longer time. The contact angle changes during the evolution process as well.
\begin{figure}[!h]
	\center
	\includegraphics[scale=0.5]{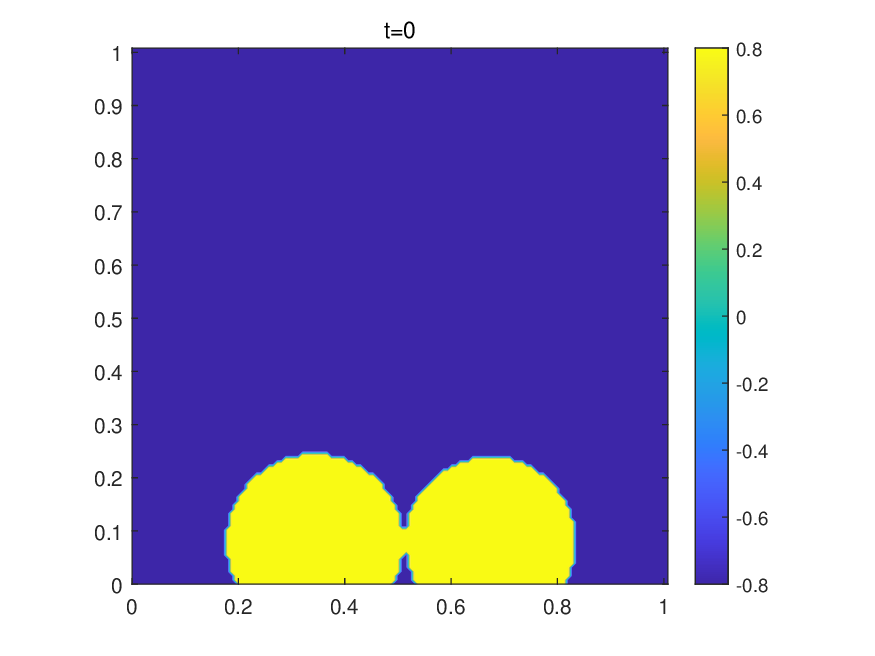}
	\caption{Initial profile of the fusion process of two circular droplets.}  
	\label{fig:ini two droplets}
\end{figure}

\begin{figure}[!h]
	\center
	\subfigure[]{
		\hspace{-10mm}\includegraphics[scale=0.38]{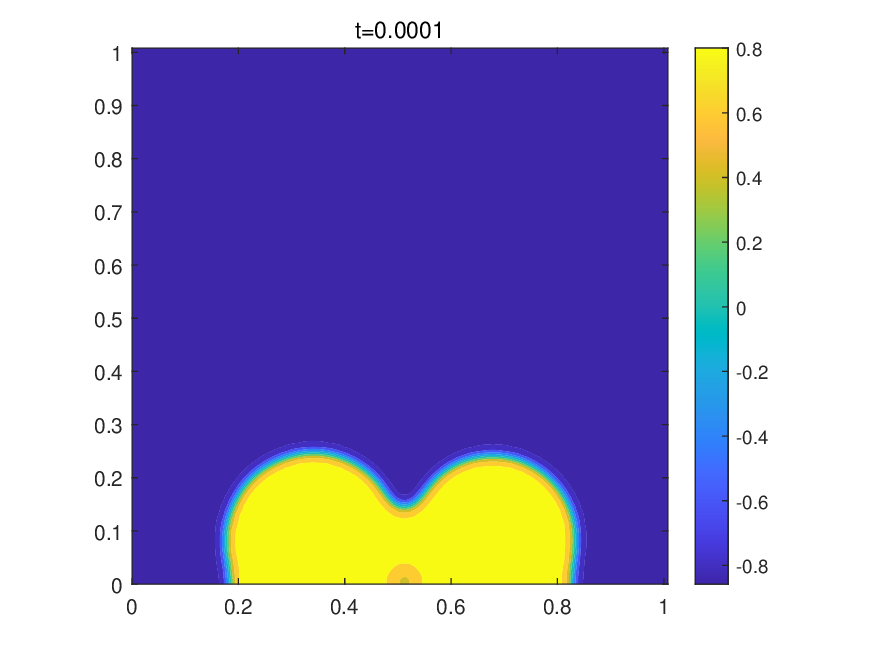}
		\hspace{-14mm}\includegraphics[scale=0.38]{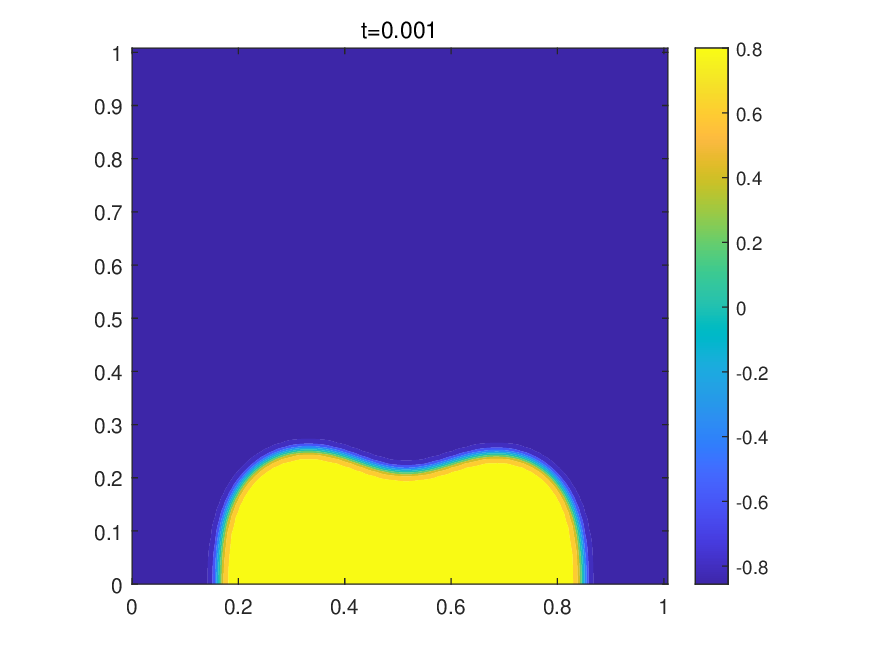}
		\hspace{-14mm}\includegraphics[scale=0.38]{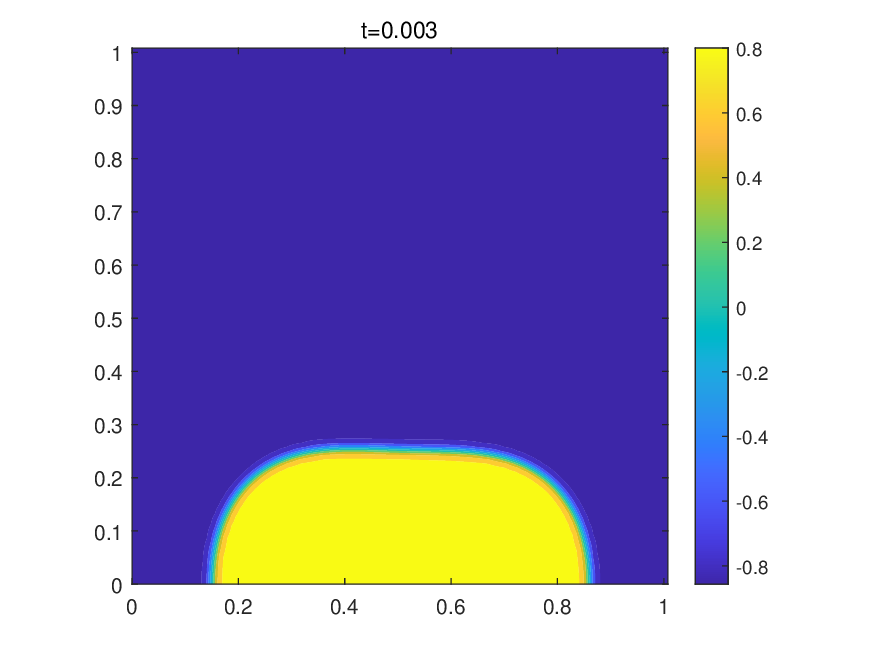}
		\hspace{-14mm}\includegraphics[scale=0.38]{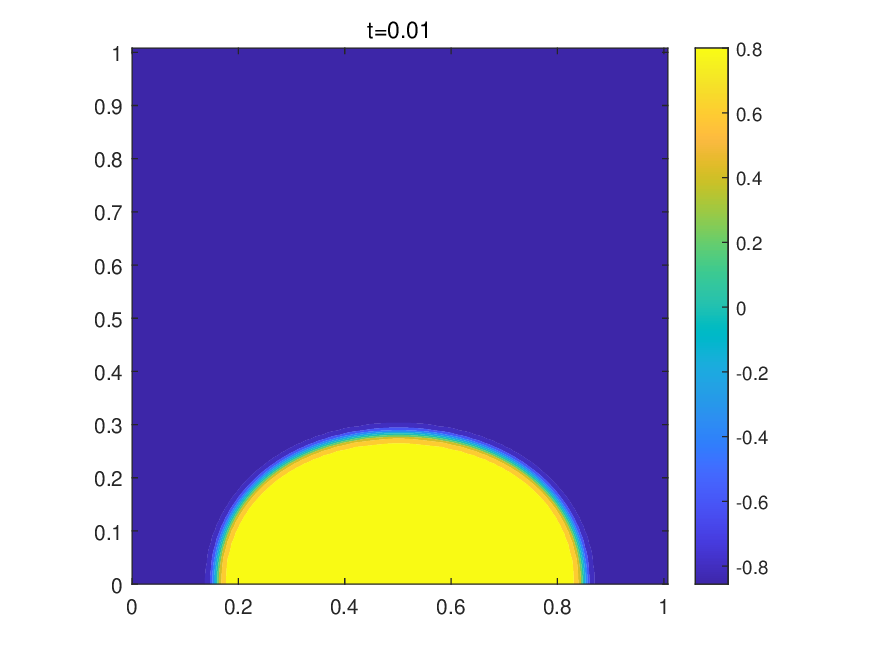}
	}

	\subfigure[]{
		\hspace{-10mm}\includegraphics[scale=0.38]{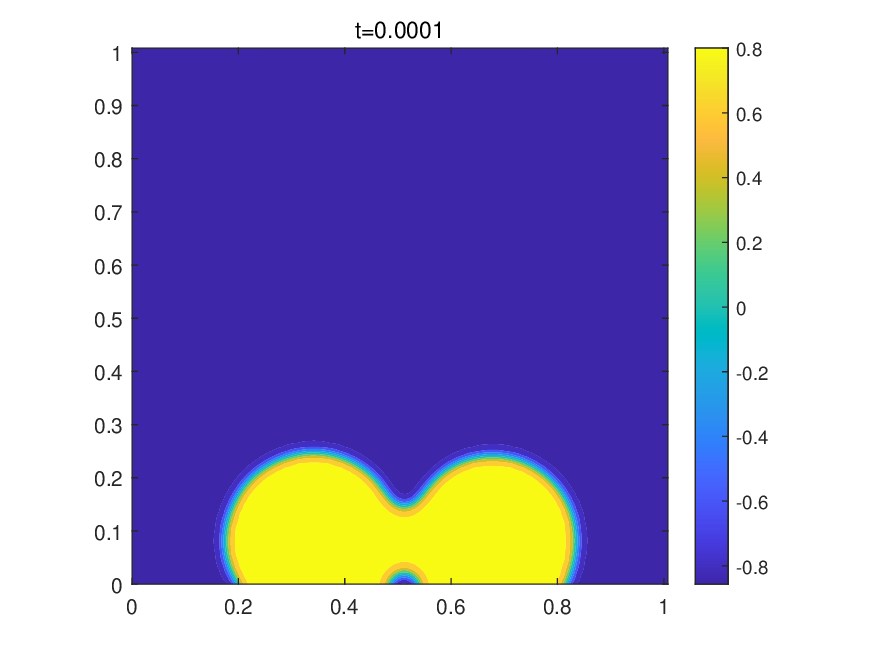}
		\hspace{-14mm}\includegraphics[scale=0.38]{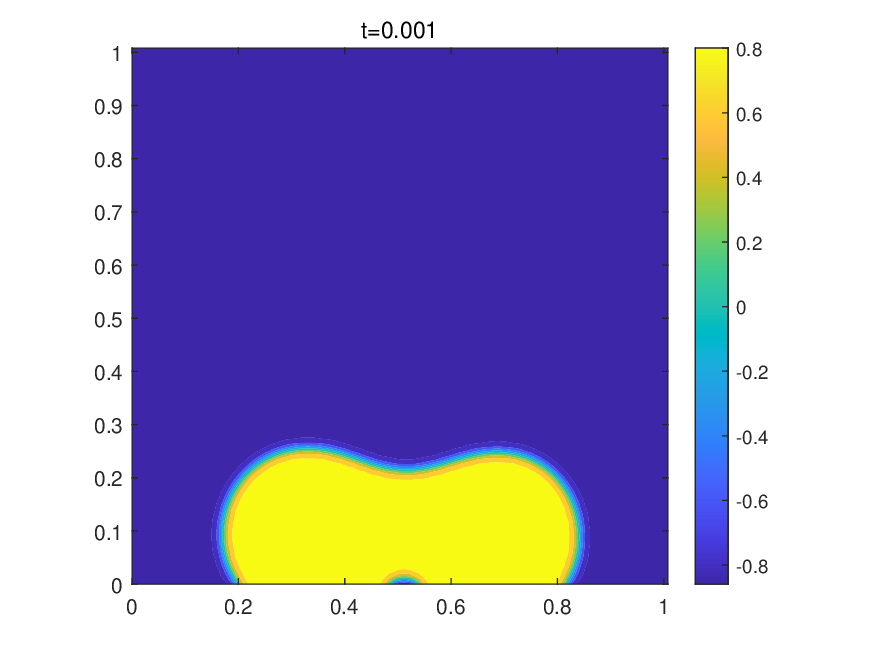}
		\hspace{-14mm}\includegraphics[scale=0.38]{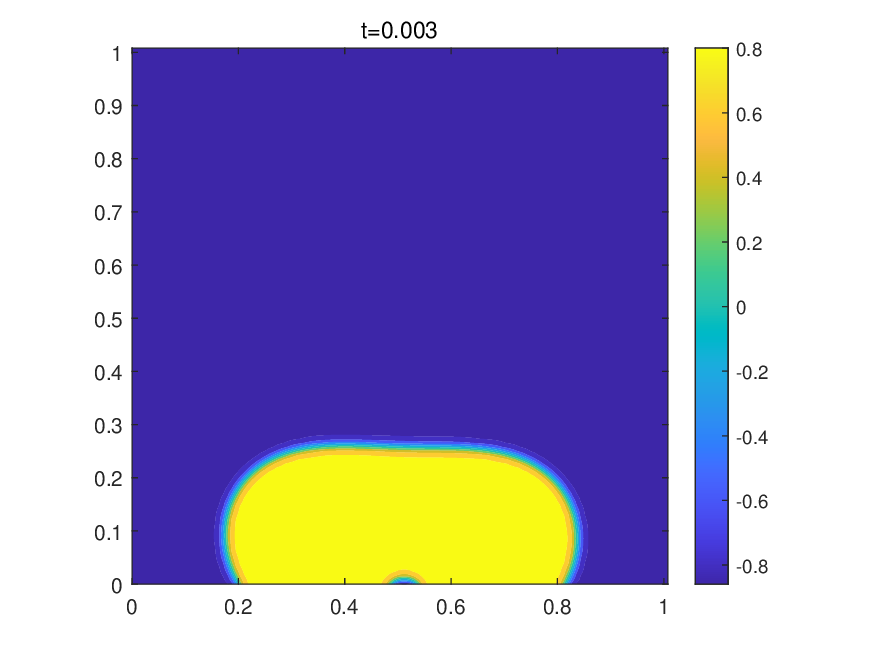}
		\hspace{-14mm}\includegraphics[scale=0.38]{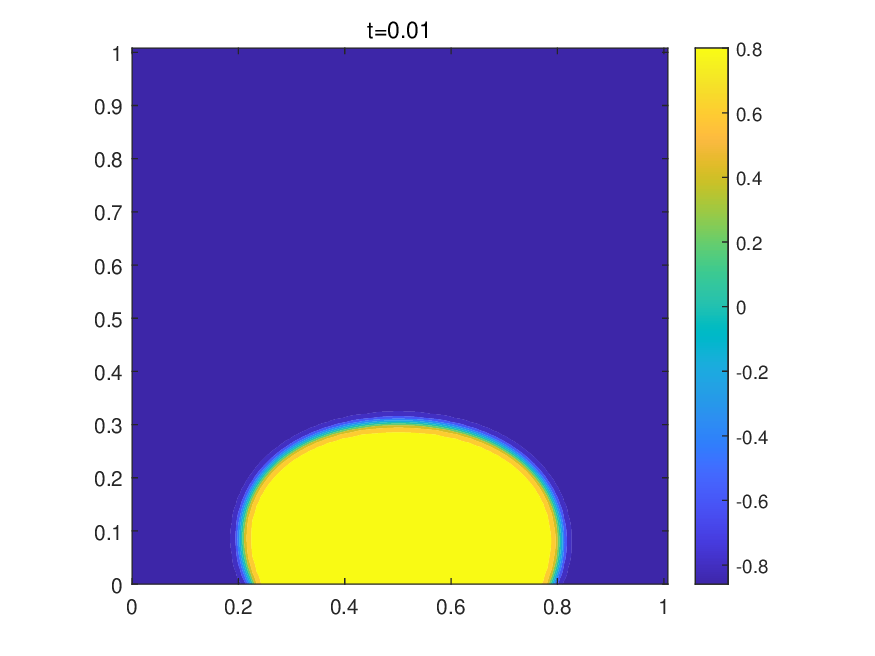}
	}
	\caption{(a): The evolution of droplet fusion process, with the homogeneous Neumann boundary condition, at time instants $t=1\times 10^{-4}$, $1\times 10^{-3}$, $3\times 10^{-3}$, $1\times 10^{-2}$. (b) The evolution of droplet fusion process, with the dynamical boundary condition at the bottom boundary, at the same time instants.}
	\label{fig: two droplets}
\end{figure}

\subsection{Comparison of evolution rate} \label{Comparison of evolution rate}
In section \ref{Two droplets evolution}, it is observed that the dynamic boundary slows down the merging of structures near the solid physical wall. Here we perform a further comparison of evolution rate between homogeneous Neumann and dynamical boundary conditions. The initial profile is presented in Figure \ref{fig:ini comparison} and the simulation results are displayed in Figure \ref{fig: two down}. Once again, the physical boundary is at the bottom.

More obviously than Figure~\ref{fig: two droplets}, the blue part surrounded by yellow gradually disappears and the mass conservation  constraint greatly slows down the evolution process. We have reason to believe that, if the yellow part spans a wider range, the difference in evolutionary speed will be more pronounced. Also, at equilibrium, the contact angle deviates from 90 degrees with the mass-conserving boundary condition.

\begin{figure}[!h]
	\center
	\includegraphics[scale=0.5]{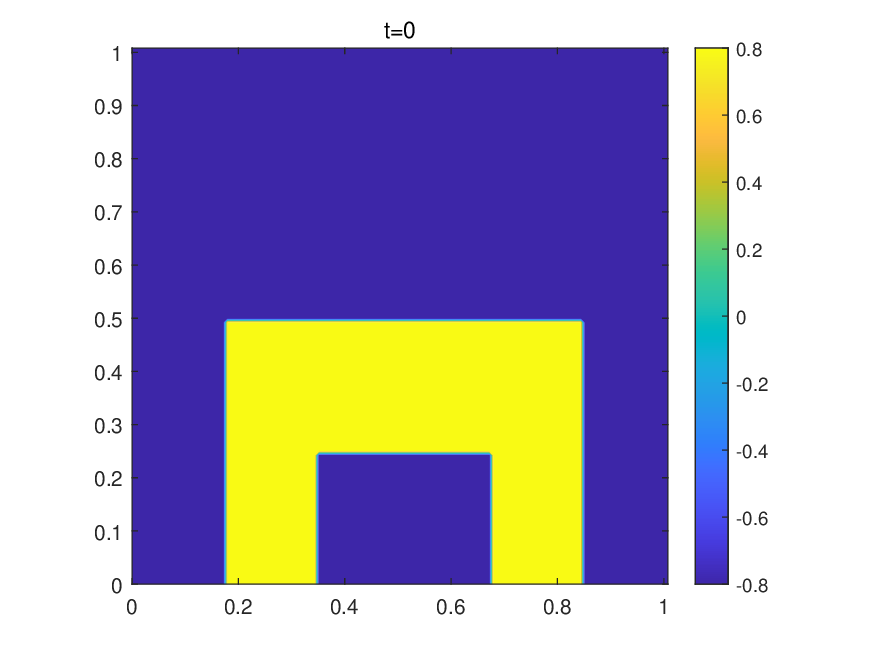}
	\caption{Initial profile of fusion process.}  
	\label{fig:ini comparison}
\end{figure}

\begin{figure}[!h]
	\center
	\subfigure[]{
		\hspace{-15mm}\includegraphics[scale=0.37]{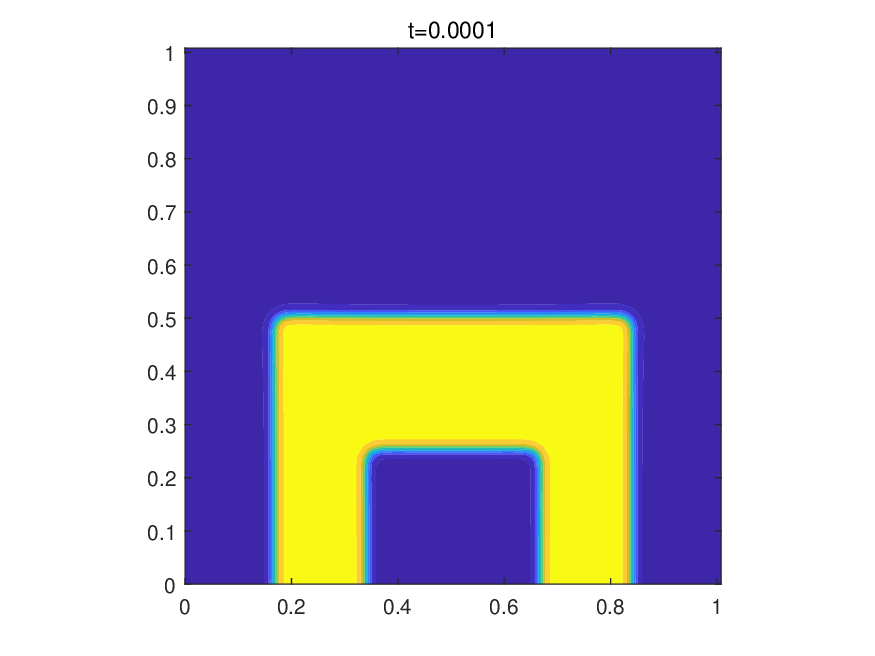}
		\hspace{-11mm}\includegraphics[scale=0.37]{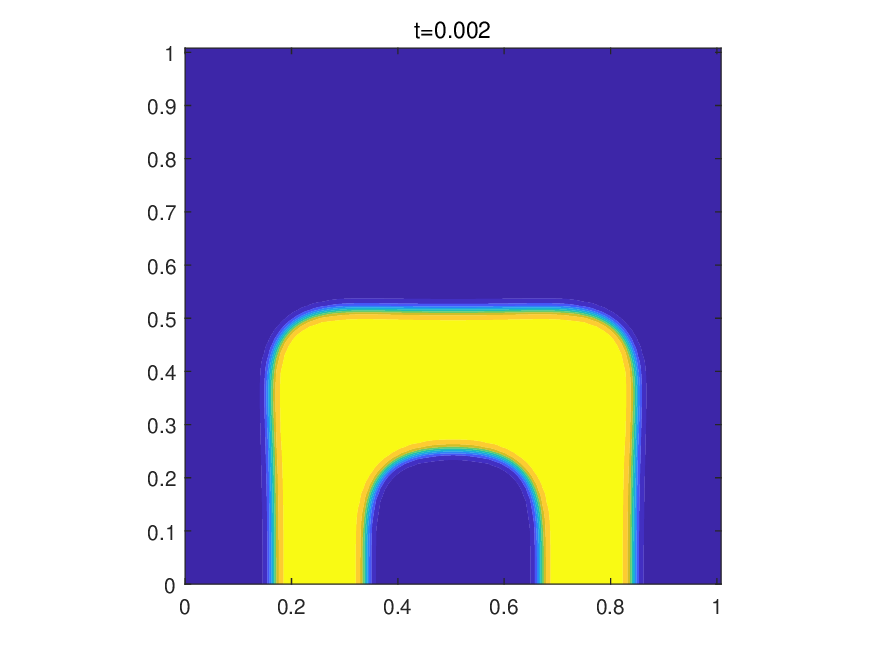}
		\hspace{-11mm}\includegraphics[scale=0.37]{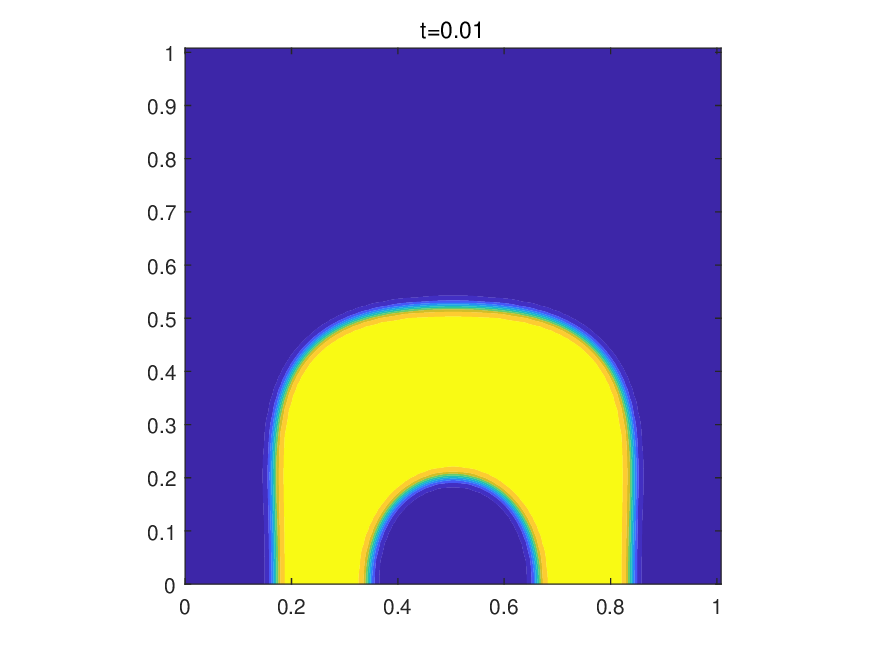}
		\hspace{-9mm}\includegraphics[scale=0.37]{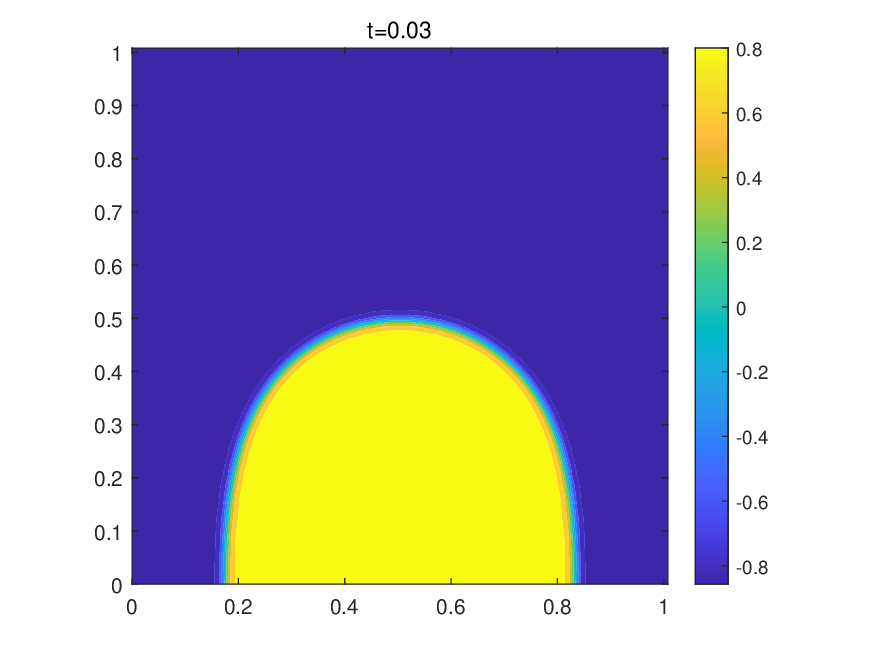}
	}
	
	\subfigure[]{
		\hspace{-15mm}\includegraphics[scale=0.37]{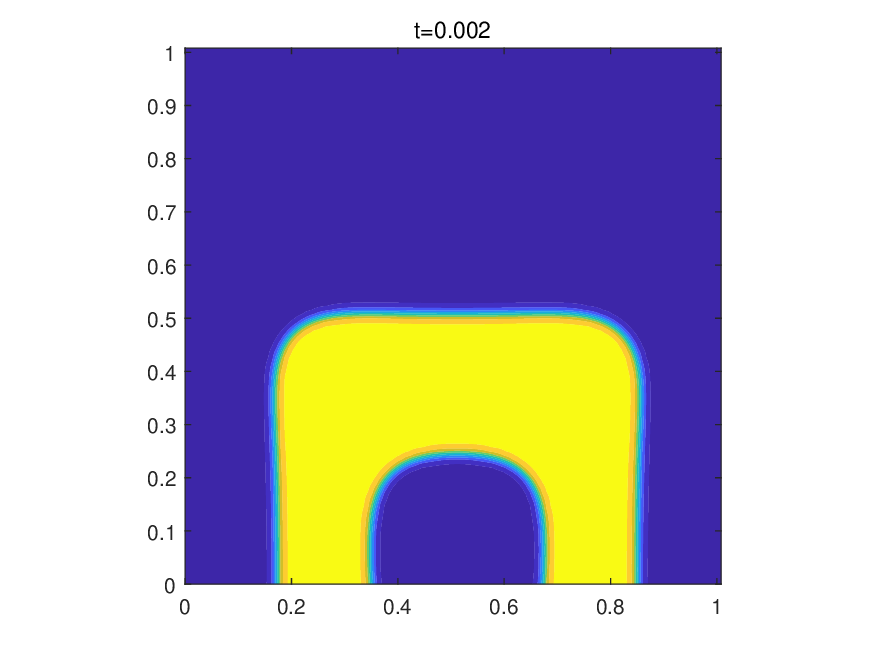}
		\hspace{-11mm}\includegraphics[scale=0.37]{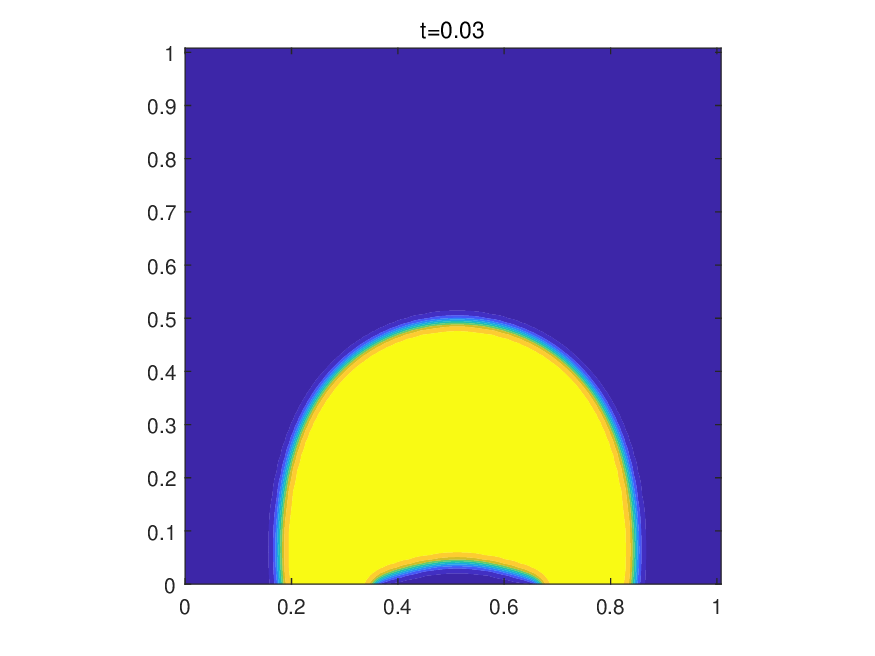}
		\hspace{-11mm}\includegraphics[scale=0.37]{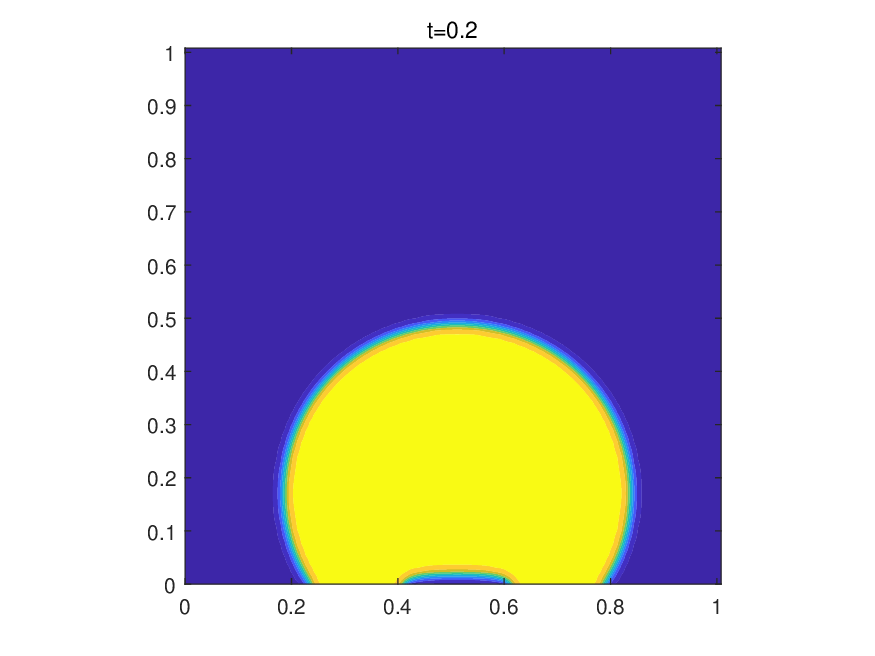}
		\hspace{-9mm}\includegraphics[scale=0.37]{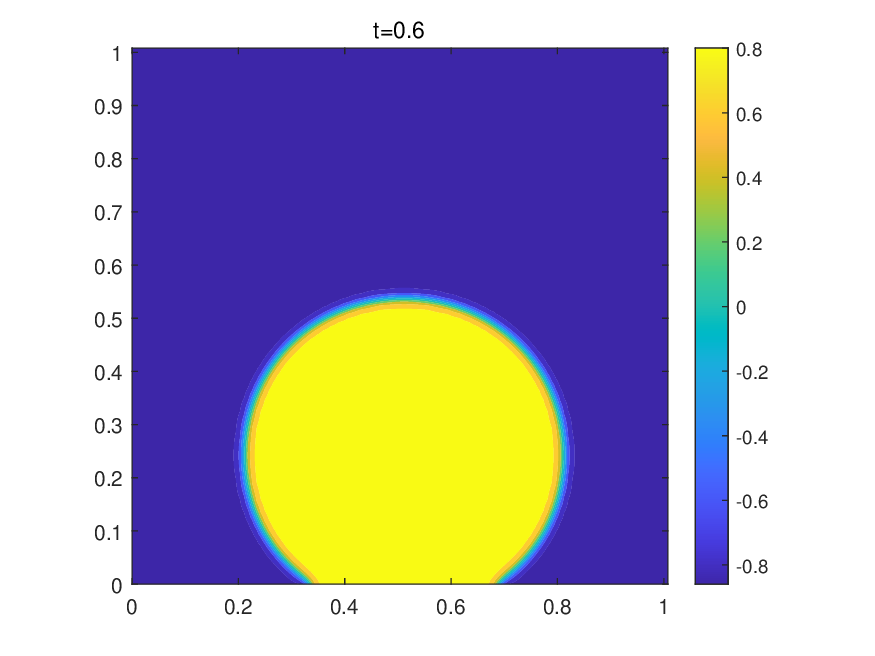}
	}
	\caption{(a): The evolution for fusion process, with the homogeneous Neumann bottom boundary condition, at time instants $t=1\times 10^{-4}$, $2\times 10^{-3}$, $1\times 10^{-2}$, $3\times 10^{-2}$. (b) The evolution for fusion process, with the dynamical bottom boundary condition, at time instants $t=2\times 10^{-3}$, $3\times 10^{-2}$, $2\times 10^{-1}$, $6\times 10^{-1}$.}
	\label{fig: two down}
\end{figure}

\section{Conclusions}  \label{sec:conclusion} 
In this paper, we have presented a fully discrete finite difference numerical scheme of the Cahn-Hilliard equation with dynamic boundary condition and Florry-Huggins energy functional. A logarithmic potential is included in the Flory-Huggins energy over the whole domain, a dynamical evolution equation for the boundary profile corresponds to a lower-dimensional singular energy potential, coupled with a non-homogeneous boundary condition for the phase variable. In the proposed numerical scheme, a convex splitting technique is applied to treat the chemical potential, at both the interior area and on the boundary section. Furthermore, the discrete boundary condition for the phase variable is coupled with the evolutionary equation of the boundary profile. The resulting  numerical system could be represented as a minimization of a discrete numerical energy functional, which contains both the interior and boundary integrals. In particular, an implicit treatment of the logarithmic term ensures the positivity-preserving property, which comes from its singular nature as the phase variable approaches the singular limit values. The total energy stability analysis could be theoretically justified as well. Some numerical experiments have been presented to demonstrate the theoretical properties of the proposed numerical scheme.

\section*{Declarations} 

C.~Wang is partially supported by the NSF DMS-2012269, DMS-2309548, S.M. Wise is partially supported by the NSF DMS-2012634, DMS-2309547, and  Z.R.~Zhang is partially supported by the NSFC No.11871105 and Science Challenge Project No. TZ2018002. 

\noindent
{\bf Conflicts of interest/Competing interests:} \, not applicable 

\noindent 
{\bf Availability of data and material:} \, not applicable 

\noindent
{\bf Code availability:} \, not applicable 

\noindent
{\bf Ethics approval:} \, not applicable 

\noindent
{\bf Consent to participate:} \, not applicable 

\noindent 
{\bf Consent for publication:} \, not applicable

\bibliographystyle{plain}
\bibliography{CHDynamicBC}

\begin{thebibliography}{10}

\bibitem{abels09b}
H.~Abels.
\newblock On a diffuse interface model for two-phase flows of viscous,
  incompressible fluids with matched densities.
\newblock {\em Arch. Ration. Mech. Anal.}, 194(2):463--506, 2009.

\bibitem{abels07}
H.~Abels and M.~Wilke.
\newblock Convergence to equilibrium for the {Cahn-Hilliard} equation with a
  logarithmic free energy.
\newblock {\em Nonlinear Anal.}, 67:3176--3193, 2007.

\bibitem{allen79}
S.~M. Allen and J.~W. Cahn.
\newblock A microscopic theory for antiphase boundary motion and its
  application to antiphase domain coursening.
\newblock {\em Acta. Metall.}, 27:1085, 1979.

\bibitem{Bao2021a}
X.~Bao and H.~Zhang.
\newblock Numerical approximations and error analysis of the {Cahn-Hilliard}
  equation with dynamic boundary conditions.
\newblock {\em Commun. Math. Sci.}, 19(3):663--685, 2021.

\bibitem{Bao2021b}
X.~Bao and H.~Zhang.
\newblock Numerical approximations and error analysis of the {Cahn-Hilliard}
  equation with reaction rate dependent dynamic boundary conditions.
\newblock {\em J. Sci. Comput.}, 87:72, 2021.

\bibitem{barrett99}
J.~Barrett and J.~Blowey.
\newblock Finite element approximation of the {Cahn-Hilliard} equation with
  concentration dependent mobility.
\newblock {\em Math. Comp.}, 68:487--517, 1999.

\bibitem{cahn1996}
J.W. Cahn, C.M. Elliott, and A.~Novick-Cohen.
\newblock The {Cahn-Hilliard} equation with a concentration dependent mobility:
  Motion by minus the {Laplacian} of the mean curvature.
\newblock {\em Europ. J. Appl. Math.}, 7:287--301, 1996.

\bibitem{cahn58}
J.W. Cahn and J.E. Hilliard.
\newblock Free energy of a nonuniform system. {I}. {I}nterfacial free energy.
\newblock {\em J. Chem. Phys.}, 28:258, 1958.

\bibitem{chen22b}
W.~Chen, J.~Jing, C.~Wang, and X.~Wang.
\newblock A positivity preserving, energy stable finite difference scheme for
  the {Flory-Huggins-Cahn-Hilliard-Navier-Stokes} system.
\newblock {\em J. Sci. Comput.}, 92:31, 2022.

\bibitem{chen22a}
W.~Chen, J.~Jing, C.~Wang, X.~Wang, and S.M. Wise.
\newblock A modified {Crank-Nicolson} scheme for the {Flory-Huggins
  Cahn-Hilliard} model.
\newblock {\em Commun. Comput. Phys.}, 31(1):60--93, 2022.

\bibitem{Chen16}
W.~Chen, Y.~Liu, C.~Wang, and S.M. Wise.
\newblock An optimal-rate convergence analysis of a fully discrete finite
  difference scheme for {Cahn-Hilliard-Hele-Shaw} equation.
\newblock {\em Math. Comp.}, 85:2231--2257, 2016.

\bibitem{chen19b}
W.~Chen, C.~Wang, X.~Wang, and S.M. Wise.
\newblock Positivity-preserving, energy stable numerical schemes for the
  {Cahn-Hilliard} equation with logarithmic potential.
\newblock {\em J. Comput. Phys.: X}, 3:100031, 2019.

\bibitem{miranville11}
L.~Cherfils, A.~Miranville, and S.~Zelik.
\newblock The {Cahn-Hilliard} equation with logarithmic potentials.
\newblock {\em Milan J. Math.}, 79:561--596, 2011.

\bibitem{elliott92a}
M.I.M. Copetti and C.M. Elliott.
\newblock Numerical analysis of the {Cahn-Hilliard} equation with a logarithmic
  free energy.
\newblock {\em Numer. Math.}, 63:39--65, 1992.

\bibitem{debussche95}
A.~Debussche and L.~Dettori.
\newblock On the {C}ahn-{H}illiard equation with a logarithmic free energy.
\newblock {\em Nonlinear Anal.}, 24:1491--1514, 1995.

\bibitem{Dong2021a}
L.~Dong, C.~Wang, S.M. Wise, and Z.~Zhang.
\newblock A positivity-preserving, energy stable scheme for a ternary
  {Cahn-Hilliard} system with the singular interfacial parameters.
\newblock {\em J. Comput. Phys.}, 442:110451, 2021.

\bibitem{Dong2022a}
L.~Dong, C.~Wang, S.M. Wise, and Z.~Zhang.
\newblock Optimal rate convergence analysis of a numerical scheme for the
  ternary {Cahn-Hilliard} system with a {Flory-Huggins-deGennes} energy
  potential.
\newblock {\em J. Comput. Appl. Math.}, 406:114474, 2022.

\bibitem{dong19b}
L.~Dong, C.~Wang, H.~Zhang, and Z.~Zhang.
\newblock A positivity-preserving, energy stable and convergent numerical
  scheme for the {Cahn-Hilliard} equation with a {Flory-Huggins-deGennes}
  energy.
\newblock {\em Commun. Math. Sci.}, 17:921--939, 2019.

\bibitem{dong20a}
L.~Dong, C.~Wang, H.~Zhang, and Z.~Zhang.
\newblock A positivity-preserving second-order {BDF} scheme for the
  {Cahn-Hilliard} equation with variable interfacial parameters.
\newblock {\em Commun. Comput. Phys.}, 28:967--998, 2020.

\bibitem{elliott96b}
C.M. Elliott and H.~Garcke.
\newblock On the {Cahn-Hilliard} equation with degenerate mobility.
\newblock {\em SIAM J. Math. Anal.}, 27:404--423, 1996.

\bibitem{Garcke20}
H.~Garcke and P.~Knopf.
\newblock Weak solutions of the cahn--hilliard system with dynamic boundary
  conditions: A gradient flow approach.
\newblock {\em SIAM Journal on Mathematical Analysis}, 52(1):340--369, 2020.

\bibitem{guo16}
J.~Guo, C.~Wang, S.M. Wise, and X.~Yue.
\newblock An {$H^2$} convergence of a second-order convex-splitting, finite
  difference scheme for the three-dimensional {Cahn-Hilliard} equation.
\newblock {\em Commun. Math. Sci.}, 14:489--515, 2016.

\bibitem{Guo24}
Y.~Guo, C.~Wang, S.M. Wise, and Z.~Zhang.
\newblock Convergence analysis of a positivity-preserving numerical scheme for
  the {Cahn-Hilliard-Stokes} system with {Flory-Huggins} energy potential.
\newblock {\em Math. Comp.}, 93(349):2185--2214, 2024.

\bibitem{Knopf2021}
P.~Knopf, K.F. Lam, C.~Liu, and S.~Metzger.
\newblock Phase-field dynamics with transfer of materials: the {Cahn-Hilliard}
  equation with reaction rate dependent dynamic boundary conditions.
\newblock {\em ESAIM: Math. Model. Numer. Anal.}, 55:229--282, 2021.

\bibitem{LiuC2021b}
C.~Liu, C.~Wang, and Y.~Wang.
\newblock A structure-preserving, operator splitting scheme for
  reaction-diffusion equations with detailed balance.
\newblock {\em J. Comput. Phys.}, 436:110253, 2021.

\bibitem{LiuC2022c}
C.~Liu, C.~Wang, and Y.~Wang.
\newblock A second order accurate, operator splitting schemes for
  reaction-diffusion systems in the energetic variational formulation.
\newblock {\em SIAM J. Sci. Comput.}, 44(4):A2276--A2301, 2022.

\bibitem{LiuC2022b}
C.~Liu, C.~Wang, Y.~Wang, and S.M. Wise.
\newblock Convergence analysis of the variational operator splitting scheme for
  a reaction-diffusion system with detailed balance.
\newblock {\em SIAM J. Numer. Anal.}, 60(2):781--803, 2022.

\bibitem{LiuC2021a}
C.~Liu, C.~Wang, S.M. Wise, X.~Yue, and S.~Zhou.
\newblock A positivity-preserving, energy stable and convergent numerical
  scheme for the {Poisson-Nernst-Planck} system.
\newblock {\em Math. Comp.}, 90:2071--2106, 2021.

\bibitem{LiuC2022a}
C.~Liu, C.~Wang, S.M. Wise, X.~Yue, and S.~Zhou.
\newblock An iteration solver for the {Poisson-Nernst-Planck} system and its
  convergence analysis.
\newblock {\em J. Comput. Appl. Math.}, 406:114017, 2022.

\bibitem{LiuC2019}
C.~Liu and H.~Wu.
\newblock An energetic variational approach for the {Cahn-Hilliard} equation
  with dynamic boundary condition: Model derivation and mathematical analysis.
\newblock {\em Arch. Rational Mech. Anal.}, 233:167--247, 2019.

\bibitem{Meng23}
X.~Meng, X.~Bao, and Z.~Zhang.
\newblock Second order stabilized semi-implicit scheme for the
  {Cahn–Hilliard} model with dynamic boundary conditions.
\newblock {\em Journal of Computational and Applied Mathematics}, 428:115145,
  2023.

\bibitem{Metzger2021}
S.~Metzger.
\newblock An efficient and convergent finite element scheme for {Cahn-Hilliard}
  equations with dynamic boundary conditions.
\newblock {\em SIAM J. Numer. Anal.}, 59(1):219--248, 2021.

\bibitem{miranville12}
A.~Miranville.
\newblock On a phase-field model with a logarithmic nonlinearity.
\newblock {\em Appl. Math.}, 57:215--229, 2012.

\bibitem{miranville04}
A.~Miranville and S.~Zelik.
\newblock Robust exponential attractors for {Cahn-Hilliard} type equations with
  singular potentials.
\newblock {\em Math. Methods Appl. Sci.}, 27:545--582, 2004.

\bibitem{QianWangZhou_JCP20}
Y.~Qian, C.~Wang, and S.~Zhou.
\newblock A positive and energy stable numerical scheme for the
  {Poisson-Nernst-Planck-Cahn-Hilliard} equations with steric interactions.
\newblock {\em J. Comput. Phys.}, 426:109908, 2021.

\bibitem{wang11a}
C.~Wang and S.M. Wise.
\newblock An energy stable and convergent finite-difference scheme for the
  modified phase field crystal equation.
\newblock {\em SIAM J. Numer. Anal.}, 49:945--969, 2011.

\bibitem{wise10}
S.M. Wise.
\newblock Unconditionally stable finite difference, nonlinear multigrid
  simulation of the {Cahn-Hilliard-Hele-Shaw} system of equations.
\newblock {\em J. Sci. Comput.}, 44:38--68, 2010.

\bibitem{wise09a}
S.M. Wise, C.~Wang, and J.~Lowengrub.
\newblock An energy stable and convergent finite-difference scheme for the
  phase field crystal equation.
\newblock {\em SIAM J. Numer. Anal.}, 47:2269--2288, 2009.

\bibitem{Yao2022}
C.~Yao, F.~Zhang, and C.~Wang.
\newblock A scalar auxiliary variable {(SAV)} finite element numerical scheme
  for the {Cahn-Hilliard-Hele-Shaw} system with dynamic boundary conditions.
\newblock {\em J. Comput. Math.}, 2022.
\newblock Accepted and in press.

\bibitem{Yuan2021a}
M.~Yuan, W.~Chen, C.~Wang, S.M. Wise, and Z.~Zhang.
\newblock An energy stable finite element scheme for the three-component
  {Cahn-Hilliard-type} model for macromolecular microsphere composite
  hydrogels.
\newblock {\em J. Sci. Comput.}, 87:78, 2021.

\bibitem{Yuan2022a}
M.~Yuan, W.~Chen, C.~Wang, S.M. Wise, and Z.~Zhang.
\newblock A second order accurate in time, energy stable finite element scheme
  for the {Flory-Huggins-Cahn-Hilliard} equation.
\newblock {\em Adv. Appl. Math. Mech.}, 14(6):1477--1508, 2022.

\bibitem{ZhangJ2021}
J.~Zhang, C.~Wang, S.M. Wise, and Z.~Zhang.
\newblock Structure-preserving, energy stable numerical schemes for a liquid
  thin film coarsening model.
\newblock {\em SIAM J. Sci. Comput.}, 43(2):A1248--A1272, 2021.

\end{thebibliography}



\end{document}